\documentclass[letterpaper,11pt]{article}

\usepackage{times}

\usepackage{amsthm,amsmath,amsfonts}

\usepackage{graphicx}
\usepackage[small]{caption}
\usepackage[hidelinks]{hyperref}
\newcommand{\orcid}[1]{\,\href{https://orcid.org/#1}{\includegraphics[width=8pt]{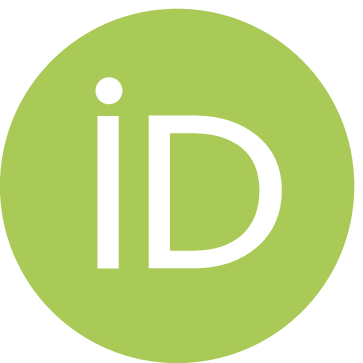}}}

\usepackage{color}

\newcommand{\C}{\mathcal{C}}
\renewcommand{\S}{\mathcal{S}}

\newtheorem{theorem}{Theorem}
\newtheorem{lemma}{Lemma}
\newtheorem{proposition}{Proposition}

\title{Caterpillars and alternating paths}

\author{Rain Jiang\orcid{0000-0002-0144-942X}\qquad
Kai Jiang\orcid{0000-0001-8165-0571}\qquad
Minghui Jiang\orcid{0000-0003-1843-9292}\,\thanks{\texttt{ dr.minghui.jiang at gmail.com}}\medskip\\
Home School, USA}

\date{}

\begin{document}

\maketitle

\begin{abstract}
	Let $p(m)$ (respectively, $q(m)$) be the maximum number $k$ such that any tree with $m$ edges
	can be transformed by contracting edges (respectively, by removing vertices)
	into a caterpillar with $k$ edges.
	We derive closed-form expressions for $p(m)$ and $q(m)$ for all $m \ge 1$.
	The two functions $p(n)$ and $q(n)$ can also be interpreted in terms of
	alternating paths among $n$ disjoint line segments in the plane,
	whose $2n$ endpoints are in convex position.
\end{abstract}

\section{Introduction}

For any family $\S$ of disjoint line segments in the plane,
an \emph{alternating path among $\S$}
is a simple polygonal chain
in which every other segment is in $\S$.
An alternating path among $\S$ is
\emph{compatible with $\S$}
if the chain does not cross any segment in $\S$ that is not in the chain.

For $n \ge 1$,
let $\hat p(n)$ (respectively, $\hat q(n)$)
be the maximum number $k$
such that any family $\S$ of $n$ disjoint segments in the plane
admits an alternating path among $\S$ (respectively, compatible with $\S$)
going through $k$ segments in $\S$.

Urrutia~\cite{Ur02} showed that
$\hat p(n) = O(n^{1/2})$
and
$\hat q(n) = O(\log n)$,
and conjectured that
$\hat p(n) = \Theta(n^{1/2})$
and
$\hat q(n) = \Theta(\log n)$.
Subsequently,
Hoffmann and T\'oth~\cite{HT03} proved that,
indeed, $\hat q(n) = \Omega(\log n)$,
and
Pach and Pinchasi~\cite{PP05} proved that
$\hat p(n) = \Omega(n^{1/3})$.
The follow-up work~\cite{HT03,PP05}
also refined Urrutia's initial upper bounds.

\begin{figure}[htbp]
	\centering{\hspace{\stretch1}\includegraphics{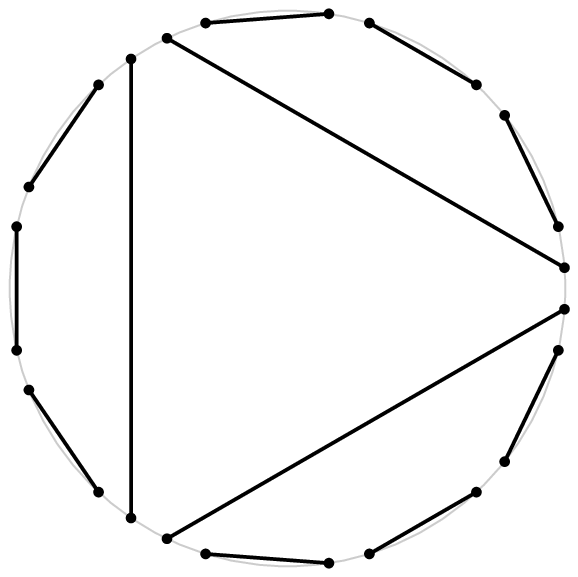}\hspace{\stretch1}\includegraphics{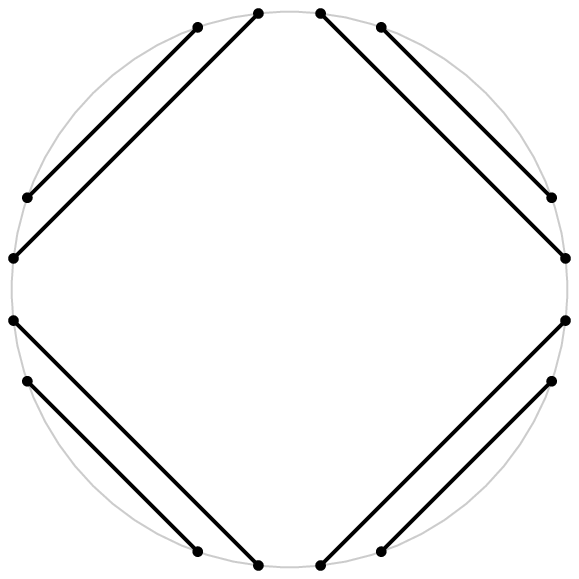}\hspace{\stretch1}}
	\caption{Left:
	a family of $3(3^k-1)/2$ segments
	in the recursive construction
	of Hoffmann and T\'oth~\cite{HT03} for $k = 2$.
	Right: a family of $2k^2$ segments, consisting of $2k$ groups of $k$ segments
	each, in the construction of Pach and Pinchasi~\cite{PP05} for $k = 2$.}\label{fig:upper}
\end{figure}

Refer to Figure~\ref{fig:upper}.
Hoffmann and T\'oth~\cite{HT03}
refined the upper bound of $\hat q(n) = O(\log n)$ to
$\hat q(n) \le 6\log_3 n + O(1)$ for all $n$,
by constructing for every $k \ge 1$ a family of $3(3^k - 1)/2$ segments
such that any alternating path compatible with these segments
can go through at most $6k-3$ segments.
Pach and Pinchasi~\cite{PP05}
refined the upper bound of $\hat p(n) = O(n^{1/2})$ to
$\hat p(n) \le \sqrt{8n} - 2$ for $n = 2k^2$,
by constructing for every $k \ge 1$ a family of $2k^2$ segments
such that any alternating path among these segments
can go through at most $4k-2$ segments;
they remarked that ``it seems likely that
the order of magnitude of this bound is not far from optimal.''

In both constructions~\cite{HT03,PP05}, as illustrated in Figure~\ref{fig:upper},
the endpoints of the segments are in convex position.
It is natural to ask whether these upper bounds~\cite{HT03,PP05}
are close to optimal at least in this restricted setting.

For $n \ge 1$,
let $p(n)$ (respectively, $q(n)$)
be the maximum number $k$
such that any family $\S$ of $n$ disjoint segments in the plane,
whose $2n$ endpoints are in convex position,
admits an alternating path among $\S$ (respectively, compatible with $\S$)
going through $k$ segments in $\S$.

We clearly have $\hat p(n) \le p(n)$ and $\hat q(n) \le q(n)$ for all $n \ge 1$.
Effectively,
Pach and Pinchasi~\cite{PP05}
proved that
$p(n) \le \sqrt{8n} - 2$ for $n = 2k^2$,
and
Hoffmann and T\'oth~\cite{HT03}
proved that
$q(n) \le 6\log_3 n + O(1)$ for all $n \ge 1$.

Our two theorems in the following imply that
the upper bound on $p(n)$ by
Pach and Pinchasi~\cite{PP05}
is indeed optimal for $n = 2k^2$,
and that the upper bound on $q(n)$ by
Hoffmann and T\'oth~\cite{HT03}
is best possible apart from an additive constant for all $n \ge 1$:

\begin{theorem}\label{thm:conv-sqrt}
	For $n \ge 1$,
	$p(n) = \lceil \sqrt{8n} - 2 \rceil$.
\end{theorem}

\begin{theorem}\label{thm:conv-log}
	For $1 \le n \le 170$,
	\begin{center}
		\begin{tabular}{c|c c c c c c c c c}
			\hline
			$n$ & $[1,4]$ & $[5,6]$ & $[7,8]$ & $[9,10]$ & $[11,12]$ & $[13,15]$ & $[16,20]$ & $[21,25]$ & $[26,30]$\\
			$q(n)$ & $n$ & $5$ & $6$ & $7$ & $8$ & $9$ & $10$ & $11$ & $12$\\
			\hline
			$n$ & $[31,35]$ & $[36,44]$ & $[45,55]$ & $[56,66]$ & $[67,80]$ & $[81,96]$ & $[97,115]$ & $[116,138]$ & $[139,170]$\\
			$q(n)$ & $13$ & $14$ & $15$ & $16$ & $17$ & $18$ & $19$ & $20$ & $21$\\
			\hline
		\end{tabular}
	\end{center}
	For $n \ge 171$,
	$q(n) = \max\{\, q_r(n) \mid 0 \le r \le 5 \,\}$,
	where
	\begin{align*}
		q_0(n) &= 6\left\lfloor\frac16\left\lceil
		6\log_3\left( \frac{2n}{55} + \frac1{11} \right) + 11
		\right\rceil\right\rfloor\\
		q_1(n) &= 6\left\lfloor\frac16\left\lceil
		6\log_3\left( \frac{n}{33} + \frac1{11} \right) + 11
		\right\rceil\right\rfloor + 1\\
		q_2(n) &= 6\left\lfloor\frac16\left\lceil
		6\log_3\left( \frac{2n}{235} + \frac1{47} \right) + 17
		\right\rceil\right\rfloor + 2\\
		q_3(n) &= 6\left\lfloor\frac16\left\lceil
		6\log_3\left( \frac{n}{141} + \frac1{47} \right) + 17
		\right\rceil\right\rfloor + 3\\
		q_4(n) &= 6\left\lfloor\frac16\left\lceil
		6\log_3\left( \frac{2n}{115} + \frac1{23} \right) + 11
		\right\rceil\right\rfloor + 4\\
		q_5(n) &= 6\left\lfloor\frac16\left\lceil
		6\log_3\left( \frac{n}{69} + \frac1{23} \right) + 11
		\right\rceil\right\rfloor + 5.
	\end{align*}
\end{theorem}

Recall that ``a caterpillar is a tree which metamorphoses into a path when its cocoon
of endpoints is removed''~\cite{HS73}.
In other words, a \emph{caterpillar} is a tree whose non-leaf vertices induce a path.
Our proofs of Theorem~\ref{thm:conv-sqrt} and Theorem~\ref{thm:conv-log}
use the following equivalent characterizations of $p(n)$ and $q(n)$
in terms of caterpillars and trees.
From this perspective,
the closed-form expressions for $p(n)$ and $q(n)$ that we obtained may be of interest
beyond the topic of alternating paths through disjoint segments.

\begin{proposition}\label{prp:eq}
	For $m \ge 1$,
	$p(m)$ (respectively, $q(m)$)
	is the maximum number $k$
	such that any tree with $m$ edges
	can be transformed into a caterpillar with $k$ edges
	by contracting edges
	(respectively, by removing vertices).
\end{proposition}

Throughout the paper, the \emph{size} of a graph refers to the number of edges in it.

\section{Connection between alternating paths and caterpillars in trees}

In this section we prove Proposition~\ref{prp:eq}
by showing that the two functions $p(n)$ and $q(n)$ for $n \ge 1$
can be interpreted in terms of caterpillars and trees.

\begin{figure}[htbp]
	\begin{center}
		\hspace*{\stretch1}\includegraphics{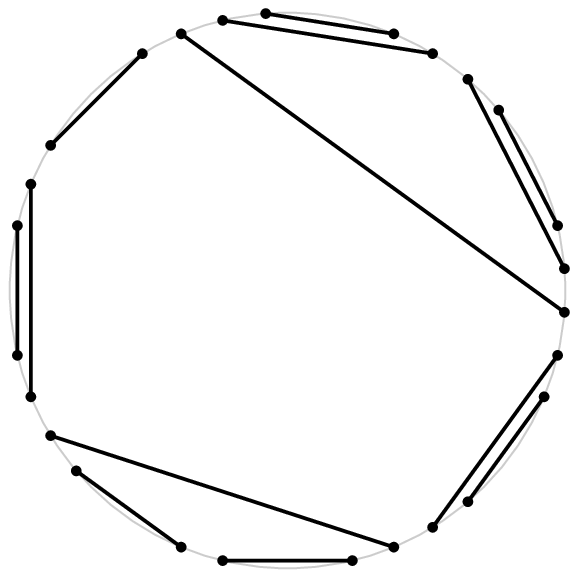}\hspace*{\stretch1}\includegraphics{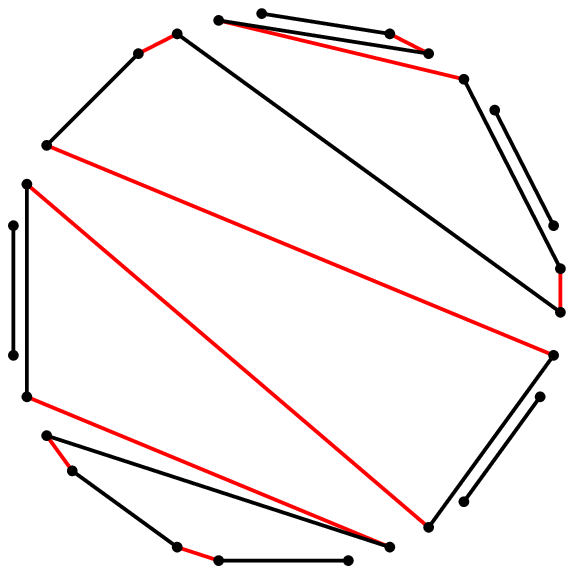}\hspace*{\stretch1}
	\end{center}
	\begin{center}
		\hspace*{\stretch1}\includegraphics{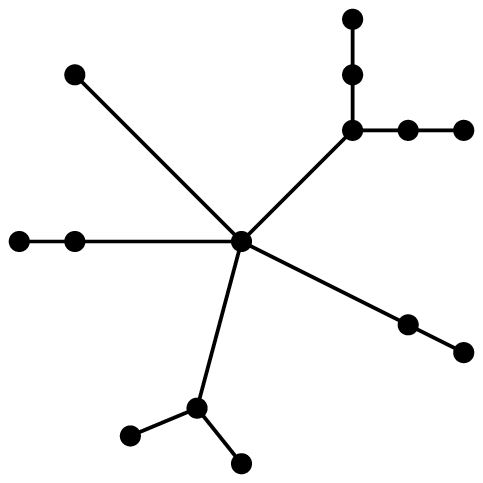}\hspace*{\stretch1}\includegraphics{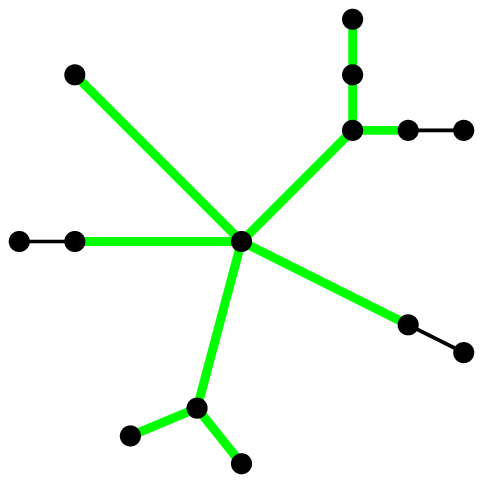}\hspace*{\stretch1}
	\end{center}
	\caption{Top left: a family $\S$ of $13$ disjoint segments as chords of a convex body $C(\S)$.
	Top right: an alternating path compatible with $\S$ going through $10$ segments.
	Bottom left: a tree $T(\S)$ with $13$ edges corresponding to $\S$.
	Bottom right: a caterpillar with $10$ edges in $T(\S)$.}\label{fig:eq}
\end{figure}

Refer to Figure~\ref{fig:eq}.
Let $\S$ be a family of $n$ disjoint segments in the plane,
whose $2n$ endpoints are in convex position.
Then there exists a convex body $C(\S)$ with the $2n$ endpoints on its boundary,
and with the $n$ segments as chords,
which partition $C(\S)$ into $n+1$ cells.
The intersection graph
with $n+1$ vertices for the $n+1$ cells,
and with two vertices connected by an edge if and only if
the corresponding cells have one of the $n$ segments as their shared boundary,
is a tree, which we denote by $T(\S)$.

Conversely,
for any tree $T$ with $n$ edges,
it is easy to construct a family $\S$ of $n$ disjoint segments
as chords of a convex body $C$
such that the intersection graph of the resulting $n+1$ cells of $C$
is $T$.

The following lemma shows that
alternating paths compatible with $\S$
are equivalent to
caterpillars obtained from $T(\S)$ by removing vertices:

\begin{lemma}\label{lem:eq1}
	$\S$ admits an alternating path compatible with $\S$ going through $k$ segments
	if and only if
	$T(\S)$ contains
	a caterpillar with $k$ edges
	as a subgraph.
\end{lemma}

\begin{proof}
	The lemma is trivially true when $k = 1$. Now assume that $k \ge 2$.
	We first prove the direct implication.
	Suppose that there is an alternating path compatible with $\S$
	that goes through a sequence $\S'$ of $k$ segments in $\S$.
	Since every segment bordering two adjacent cells of $C(\S)$
	corresponds to an edge between two vertices in $T(\S)$,
	and every two consecutive segments in $\S'$
	are on the boundary of some common cell,
	it follows that
	every maximal subsequence of consecutive segments in $\S'$
	around a common cell in $C(\S)$
	corresponds to a star of edges incident to a common vertex in $T(\S)$.
	Moreover,
	every three consecutive segments in $\S'$ not around a common cell
	must span two adjacent cells
	with the middle segment as their shared boundary,
	which corresponds to an edge connecting two stars in $T(\S)$.
	Thus the $k$ segments in $\S'$
	correspond to a path of stars in $T(\S)$,
	which is a caterpillar with $k$ edges.

	We next prove the reverse implication.
	Suppose that $T(\S)$ contains a caterpillar with $k$ edges,
	whose path of non-leaf vertices corresponds to
	a sequence $\C$ of cells of $C(\S)$.
	Then the edges of the caterpillar incident to each non-leaf vertex
	correspond to a subfamily of segments in $\S$
	on the boundary of some common cell in $\C$,
	and each edge in the non-leaf path corresponds
	to a segment in $\S$ bordering two adjacent cells in $\C$.
	Since the cells of $C(\S)$ are all convex,
	the disjoint segments on the boundary of each cell
	can be linked into a non-crossing path within the cell,
	starting and ending at any two endpoints of any two segments.
	Then along the sequence $\C$ of cells, these paths
	can be concatenated into a single alternating path compatible with $\S$.
\end{proof}

The following lemma shows another equivalence between
alternating paths among $\S$ and caterpillars obtained from $T$ by contracting edges:

\begin{lemma}\label{lem:eq2}
	$\S$ admits an alternating path among $\S$ going through $k$ segments
	if and only if
	a caterpillar with $k$ edges
	can be obtained from $T(\S)$ by contracting edges.
\end{lemma}

\begin{proof}
	The deletion of a segment in $\S$ corresponds to
	the merging of two adjacent cells of $C(\S)$,
	and to the contraction of an edge in $T(\S)$.
	For any subfamily $\S'$ of $k$ segments in $\S$,
	any alternating path among $\S$ that goes through all $k$ segments of $\S'$
	is an alternating path compatible with $\S'$,
	and conversely,
	any alternating path compatible with $\S'$ is an alternating path among $\S$.
	By Lemma~\ref{lem:eq1},
	$\S'$ admits an alternating path compatible with $\S'$ going through all $k$ segments
	if and only if 
	$T(\S')$ is a caterpillar with $k$ edges.
\end{proof}

Recall that for $n \ge 1$,
we defined $p(n)$ (respectively, $q(n)$)
as the maximum number $k$
such that any family $\S$ of $n$ disjoint segments in the plane
whose endpoints are in convex position
admits an alternating path among $\S$ (respectively, compatible with $\S$)
going through $k$ segments in $\S$.
By Lemma~\ref{lem:eq1} and Lemma~\ref{lem:eq2},
it follows that for $m \ge 1$,
$p(m)$ (respectively, $q(m)$)
is the maximum number $k$
such that any tree with $m$ edges
can be transformed into a caterpillar with $k$ edges
by contracting edges
(respectively, by removing vertices).
This completes the proof of Proposition~\ref{prp:eq}.

\bigskip
In the next two sections,
we prove Theorem~\ref{thm:conv-sqrt} and Theorem~\ref{thm:conv-log}
in terms of caterpillars and trees.
By convention, we will use the symbol $m$ to denote the number of edges in a graph,
and determine $p(m)$ and $q(m)$ for $m \ge 1$.

\section{Contracting a tree into a caterpillar}

In this section we prove Theorem~\ref{thm:conv-sqrt} that
for $m \ge 1$,
$p(m) = \lceil \sqrt{8m} - 2 \rceil$.
Recall that
$p(m)$
is the maximum number $k$
such that any tree with $m$ edges
can be transformed into a caterpillar with $k$ edges
by contracting edges.
We first prove three lemmas.

For any tree $T$, denote by $\kappa(T)$
the number of leaves in $T$,
plus the diameter of $T$, minus $2$.

\begin{lemma}\label{lem:kappa}
	For any tree $T$,
	the maximum size of a caterpillar
	that can be obtained from $T$ by contracting edges
	is $\kappa(T)$.
	Moreover, for any $1 \le k \le \kappa(T)$,
	a caterpillar of size $k$ can be obtained from $T$
	by contracting edges.
\end{lemma}

\begin{proof}
	When $\kappa(T) = 1$, $T$ must be a tree with a single edge.
	The lemma obviously holds in this case. Now assume that $\kappa(T) \ge 2$.
	Let $T'$ be a caterpillar of the maximum size obtained from a tree $T$ by contracting edges.
	By the maximality of $T'$, each leaf of $T'$ must be a leaf of $T$.
	Moreover, the edges of the non-leaf path of $T'$, if any, must come from
	some non-leaf path in $T$,
	whose length is at most the diameter of $T$ minus $2$.
	Thus the number of edges in $T'$ is at most the number of leaves in $T$
	plus the diameter of $T$ minus $2$, which is $\kappa(T)$.
	This number is actually attainable: simply keep all leaf edges
	and all edges of a diameter path, and contract the rest.
	From such a caterpillar with $\kappa(T)$ edges,
	a caterpillar with $k < \kappa(T)$ edges
	can then be obtained by contracting any $\kappa(T) - k$ edges.
\end{proof}

For $k \ge 1$,
denote by $e(k)$ the maximum size of a tree $T$ with
$\kappa(T) = k$.
For $d \ge 2$ and $l \ge 2$,
denote by $e(d,l)$ the maximum size of a tree
with diameter $d$ and with $l$ leaves.

A \emph{spider} is a tree
with at most one vertex of degree greater than two.
In the next two lemmas, we determine $e(d,l)$ and $e(k)$ by showing that
the extremal cases are realized by spiders.

\begin{lemma}\label{lem:edl}
	For $d \ge 2$ and $l \ge 2$,
	$e(d,l) = \frac12\cdot d\cdot l$ when $d$ is even,
	$e(d,l) = \frac12\cdot (d-1)\cdot l + 1$ when $d$ is odd,
	and there exists a spider $R_{d,l}$ with diameter $d$, $l$ leaves, and $e(d,l)$ edges.
\end{lemma}

\begin{proof}
	Consider two cases for any tree with diameter $d$, $l$ leaves, and $m$ edges:
\begin{enumerate}\setlength\itemsep{0pt}

		\item
			$d$ is even.
			Then there is an internal node $v$ in the tree whose distance to any other
			node is at most $d/2$.
			The number $m$ of edges in the tree is maximized when
			the paths from $v$ to the $l$ leaves all have the same length $d/2$
			and are edge-disjoint,
			that is, when the tree is a spider $R_{d,l}$ centered at $v$
			with $l$ legs of equal length $d/2$.
			Thus $m \le \frac12\cdot d\cdot l$.

		\item
			$d$ is odd.
			Then there is an edge $\{u,v\}$
			between two internal nodes $u$ and $v$ in the tree
			such that the distance from this edge to any other node
			is at most $(d-1)/2$.
			The number $m$ of edges in the tree is maximized when
			the shortest paths from $\{u,v\}$ to the $l$ leaves
			all have the same length $(d-1)/2$ and are edge-disjoint.
			Thus $m \le \frac12\cdot (d-1)\cdot l + 1$.
			There are multiple extremal cases when $l > 2$.
			In particular, if one of the $l$ leaves is on one side of the edge $\{u,v\}$,
			and the other $l- 1$ leaves are all on the other side,
			then the tree is a spider $R_{d,l}$
			with one leg of length $(d+1)/2$ and $l-1$ legs of length $(d-1)/2$.
			\qedhere

\end{enumerate}
\end{proof}

\begin{lemma}\label{lem:ek}
	For $k \ge 1$,
	$e(k) = \frac18k(k+4)$ when $k \bmod 4 = 0$,
	$e(k) = \frac18(k+2)^2$ when $k \bmod 4 = 2$,
	$e(k) = \frac18(k+1)(k+3)$ when $k \bmod 2 = 1$,
	and there exists a spider $R_k$ with $\kappa(R_k) = k$ and with $e(k)$ edges.
	In particular, $\frac18(k+2)^2 - \frac12 \le e(k) \le \frac18(k+2)^2$,
	and $e(k)$ is strictly increasing for $k \ge 1$.
\end{lemma}

\begin{proof}
	The only tree $T$ with $\kappa(T) = 1$ is $P_2$, a path with a single edge.
	Thus $e(1) = 1$, which equals $\frac18(1+1)(1+3)$,
	and we have $R_1 = P_2$.
	Also, the only tree $T$ with $\kappa(T) = 2$ is $P_3$, a path with two edges.
	Thus $e(2) = 2$, which equals $\frac18(2+2)^2$,
	and we have $R_2 = P_3$.
	For $k \ge 3$, we have
	$$
	e(k) = \max\{\, e(d,l) \mid d \ge 2,\, l \ge 2,\, l + d - 2 = k \,\}.
	$$
	To determine $e(k)$, we evaluate the maxima of $e(d,l)$
	for even $d$ and for odd $d$ separately,
	and apply Lemma~\ref{lem:edl}.
	Consider four cases:
\begin{itemize}\setlength\itemsep{0pt}

		\item
			$k \bmod 4 = 2$.
			Then
			\begin{align*}
				\max_{d \textrm{ even}} e(d,l)
				&= e\left( \frac{k+2}2, \frac{k+2}2 \right)
				= \frac{(k+2)^2}8,
				\\
				\max_{d \textrm{ odd}} e(d,l)
				&= e\left( \frac{k+4}2, \frac{k}2 \right)
				= \frac{k(k+2)}8 + 1.
			\end{align*}
			Since $\frac18(k+2)^2 - \frac18 k(k+2) - 1 = \frac14(k+2) - 1 \ge 1$
			for $k \ge 6$,
			we have
			$e(k) = \frac18(k+2)^2$.

		\item
			$k \bmod 4 = 0$.
			Then
			\begin{align*}
				\max_{d \textrm{ even}} e(d,l)
				&= e\left( \frac{k}2, \frac{k+4}2 \right)
				= e\left( \frac{k+4}2, \frac{k}2 \right)
				= \frac{k(k+4)}8,
				\\
				\max_{d \textrm{ odd}} e(d,l)
				&= e\left( \frac{k+2}2, \frac{k+2}2 \right)
				= \frac{k(k+2)}8 + 1.
			\end{align*}
			Since $\frac18 k(k+4) - \frac18 k(k+2) - 1 = \frac14 k - 1 \ge 0$ for $k \ge 4$,
			we have
			$e(k) = \frac18 k(k+4) = \frac18(k+2)^2 - \frac12$.

		\item
			$k \bmod 4 = 3$.
			Then
			\begin{align*}
				\max_{d \textrm{ even}} e(d,l)
				&= e\left( \frac{k+1}2, \frac{k+3}2 \right)
				= \frac{(k+1)(k+3)}8,
				\\
				\max_{d \textrm{ odd}} e(d,l)
				&= e\left( \frac{k+3}2, \frac{k+1}2 \right)
				= \frac{(k+1)^2}8 + 1.
			\end{align*}
			Since $\frac18(k+1)(k+3) - \frac18(k+1)^2 - 1 = \frac14(k+1) - 1 \ge 0$
			for $k \ge 3$,
			we have
			$e(k) = \frac18(k+1)(k+3) = \frac18(k+2)^2 - \frac18$.

		\item
			$k \bmod 4 = 1$.
			Then
			\begin{align*}
				\max_{d \textrm{ even}} e(d,l)
				&= e\left( \frac{k+3}2, \frac{k+1}2 \right)
				= \frac{(k+1)(k+3)}8,
				\\
				\max_{d \textrm{ odd}} e(d,l)
				&= e\left( \frac{k+1}2, \frac{k+3}2 \right)
				= e\left( \frac{k+5}2, \frac{k-1}2 \right)
				= \frac{(k-1)(k+3)}8 + 1.
			\end{align*}
			Since $\frac18(k+1)(k+3) - \frac18(k-1)(k+3) - 1 = \frac14(k+3) - 1 \ge 1$
			for $k \ge 5$,
			we again have
			$e(k) = \frac18(k+1)(k+3) = \frac18(k+2)^2 - \frac18$.

	\end{itemize}
	The last two cases
	$k \bmod 4 = 3$
	and
	$k \bmod 4 = 1$
	can be combined into a single case $k \bmod 2 = 1$.
	In any case,
	there exists by Lemma~\ref{lem:edl}
	a spider $R_k = R_{d,l}$ with $\kappa(R_k) = \kappa(R_{d,l}) = l + d - 2 = k$
	and with $e(k) = e(d,l)$ edges,
	where $d$ is even.
	In particular, we can let
	$R_3 = R_{2,3}$, $R_4 = R_{2,4}$, $R_5 = R_{4,3}$, $R_6 = R_{4,4}$,
	$R_7 = R_{4,5}$, $R_8 = R_{4,6}$, $R_9 = R_{6,5}$, $R_{10} = R_{6,6}$,
	$R_{11} = R_{6,7}$, $R_{12} = R_{6,8}$.
	Refer to Figure~\ref{fig:r} for illustrations of $R_k$ for $1 \le k \le 12$.

	Note that $\frac18(k+2)^2 - \frac12 \le e(k) \le \frac18(k+2)^2$ for all $k \ge 1$.
	It follows that
	$$
	e(k + 1) - e(k) \ge \frac18(k+3)^2 - \frac12 - \frac18(k+2)^2
	= \frac18(2k+5) - \frac12 > 0.
	$$
	Thus $e(k)$ is strictly increasing for $k \ge 1$.
\end{proof}

\begin{figure}[htbp]
	\centering\includegraphics{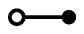}\includegraphics{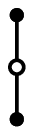}\includegraphics{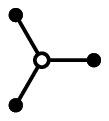}\includegraphics{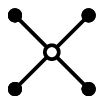}\\\includegraphics{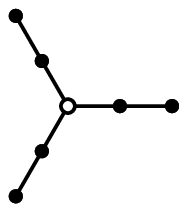}\includegraphics{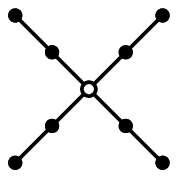}\includegraphics{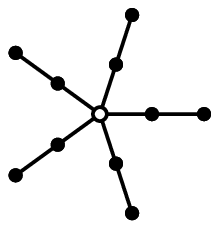}\includegraphics{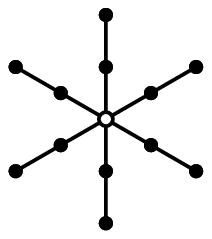}\\\includegraphics{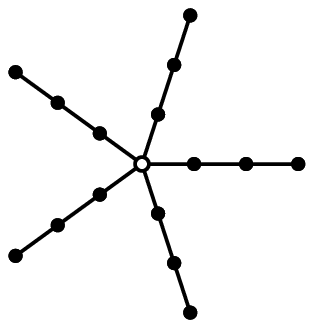}\includegraphics{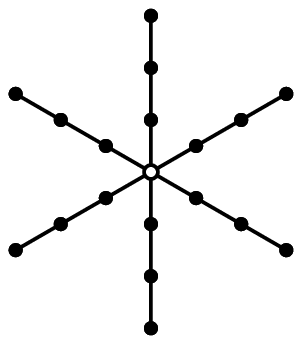}\includegraphics{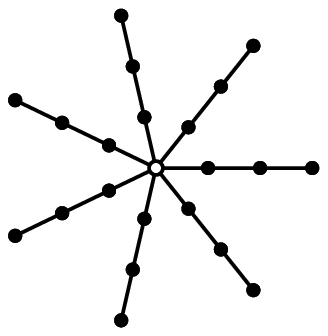}\includegraphics{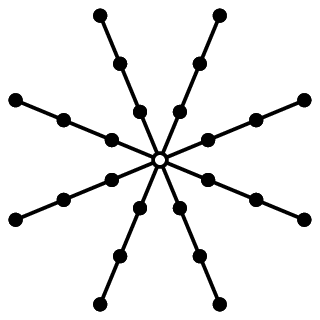}
	\caption{Spiders $R_k$ with $e(k)$ edges for $1 \le k \le 12$.}\label{fig:r}
\end{figure}

We are now ready to prove that
$p(m) = \lceil \sqrt{8m} - 2 \rceil$
for $m \ge 1$.
For $m = 1$,
we clearly have
$p(1) = \lceil \sqrt{8} - 2 \rceil = 1$.
Now fix $m \ge 2$, and let $k = \lceil \sqrt{8m} - 2 \rceil$.
Then $k \ge 2$,
and $\sqrt{8m} - 2 \le k < \sqrt{8m} - 1$.

From $k \ge \sqrt{8m} - 2$,
it follows that
$e(k) \ge \frac18(k+2)^2 - \frac12 \ge \frac18\cdot (\sqrt{8m} - 2 + 2)^2 - \frac12 = m - \frac12$,
and hence
$e(k) \ge m$
since $e(k)$ is an integer.
From $k < \sqrt{8m} - 1$,
it follows that
$e(k-1) \le \frac18(k+1)^2 < \frac18\cdot(\sqrt{8m} - 1 + 1)^2 = m$.
In summary, we have
$e(k-1) < m \le e(k)$.

By Lemma~\ref{lem:ek}, $e(k)$ is strictly increasing for $k \ge 1$.
Thus the inequality $e(k-1) < m$ implies that
any tree with $\kappa(T) \le k-1$ must have fewer than $m$ edges.
In other words,
any tree $T$ with $m$ edges must have $\kappa(T) \ge k$,
and hence can be transformed by contracting edges
into a caterpillar with $k$ edges, by Lemma~\ref{lem:kappa}.

On the other hand,
corresponding to the extremal cases of Lemma~\ref{lem:edl} and Lemma~\ref{lem:ek},
there exists a spider $R_k$ with $\kappa(R_k) = k$ and with $e(k)$ edges.
Since $\kappa(R_k) = k$,
it follows by Lemma~\ref{lem:kappa} that
$R_k$ cannot be transformed by contracting edges into a caterpillar with more than $k$ edges.
By the inequality $m \le e(k)$,
there exists a subgraph of $R_k$ with $m$ edges
which cannot be thus transformed either.
Thus
$p(m) = k = \lceil \sqrt{8m} - 2 \rceil$.
This completes the proof of Theorem~\ref{thm:conv-sqrt}.

\section{Pruning a tree into a caterpillar}

In this section we prove Theorem~\ref{thm:conv-sqrt}.
Recall that
$q(m)$
is the maximum number $k$
such that any tree with $m$ edges
can be transformed into a caterpillar with $k$ edges
by removing vertices.
In other words,
$q(m)$
is the maximum number $k$
such that any tree with $m$ edges
contains a caterpillar with $k$ edges
as a vertex-induced subgraph.

For $k \ge 1$,
let $e(k)$ be the maximum size $m$ of a tree in which
the maximum size of a caterpillar is exactly $k$.
To derive $q(m)$ for $m \ge 1$,
we first determine the exact values of $e(k)$ for $k \ge 1$.

\subsection{Single branch: $f(k)$}

In a rooted tree,
we call any caterpillar
consisting of all edges incident to a path from the root to a leaf
a \emph{very hungry} caterpillar.
For $k \ge 1$,
let $f(k)$
be the maximum size of a rooted tree in which
the root is incident to a single edge,
and the maximum size of a very hungry caterpillar is $k$.

To determine $f(k)$,
we only need to consider \emph{symmetric} trees
in which all nodes of the same depth have the same number of children.
Let $T$ be a rooted tree in which
the root is incident to a single edge,
and the maximum size of a very hungry caterpillar is $k$.
Suppose that $T$ is not symmetric.
Let $d \ge 0$ be the minimum depth at which there are nodes of different degrees in $T$.
Let $v$ be a node at depth $d$ such that the subtree of $T$ rooted at $v$ has the maximum size.
By replacing every subtree rooted at a node at depth $d$ by a copy of the subtree rooted at $v$,
we either make $T$ symmetric,
or increase the minimum depth at which there are nodes of different degrees in $T$.
Moreover, the size of $T$ does not decrease,
and the maximum size of a very hungry caterpillar in $T$ does not increase.
By repeating such replacements, we can make $T$ symmetric.
Then, by splitting the single edge incident to the root of $T$ into a path if necessary,
we can ensure that the size of $T$ does not decrease,
and the maximum size of a very hungry caterpillar in $T$ is exactly $k$.

Let $T$ be a symmetric tree of height $h \ge 1$
in which the root is incident to a single edge.
For $0 \le d \le h$,
let $c_d$ be the number of children of each node with depth $d$ in $T$,
with $c_0 = 1$ for the root,
and $c_h = 0$ for the leaves.
Then all very hungry caterpillars in $T$ have the same size
$\sum_{d = 0}^h c_d$.

We can assume that $c_1 \ge \ldots \ge c_h$ because,
if $c_d < c_{d+1}$ for some $1 \le d < h$,
then we can swap the two elements $c_d$ and $c_{d+1}$
in the sequence $\langle c_0,c_1,\ldots,c_h \rangle$,
such that the resulting sequence corresponds to a symmetric tree with
more edges between nodes of depths $d$ and $d+1$,
and with the same number of edges in other layers,
while the sum $\sum_{d = 0}^h c_d$ remains the same.

Moreover, we can assume that $c_1 \le 3$ because,
if $c_1 \ge 4$,
then we can replace the element $c_1$
in the sequence $\langle c_0,c_1,\ldots,c_h \rangle$
by two elements $c_1 - 2$ and $2$,
such that the resulting sequence of $h+2$ elements
corresponds to a symmetric tree with height $h+1$
and with more edges,
but still has the same sum.

In summary,
we can assume that $T$ is a symmetric tree of height $h \ge 1$,
with $c_0 = 1$ and $3 \ge c_1 \ge \ldots \ge c_h = 0$,
and we call such trees \emph{beautiful trees}.
Thus for $k \ge 1$,
$f(k)$
is the maximum size of a beautiful tree $T$
with $\sum_{d = 0}^h c_d = k$.

\begin{lemma}\label{lem:f}
	$f(1) = 1$,
	$f(2) = 2$,
	$f(3) = 3$,
	$f(4) = 5$,
	$f(5) = 7$.
\begin{itemize}\setlength\itemsep{0pt}
		\item
			For $k = 3i$ with $i \ge 2$,
			$f(k) = \frac12(23\cdot 3^{(k-6)/3} - 1)$.
		\item
			For $k = 3i + 1$ with $i \ge 2$,
			$f(k) = \frac12(33\cdot 3^{(k-7)/3} - 1)$.
		\item
			For $k = 3i + 2$ with $i \ge 2$,
			$f(k) = \frac12(47\cdot 3^{(k-8)/3} - 1)$.
\end{itemize}
In particular,
\,(a)\, for $k \ge 2$, $\frac{f(k)}{f(k - 1)} \ge \frac75$,
\,(b)\, for $k \ge 7$, $\frac{f(k)}{f(k - 1)} < \frac32$.
Moreover,
for $k \ge 1$, there exists a beautiful tree $B_k$ with exactly $f(k)$ edges
such that the maximum size of a very hungry caterpillar in $B_k$ is $k$,
and the maximum size of any caterpillar in $B_k$ is less than $2k$.
\end{lemma}

\begin{figure}[htbp]
	\centering\includegraphics{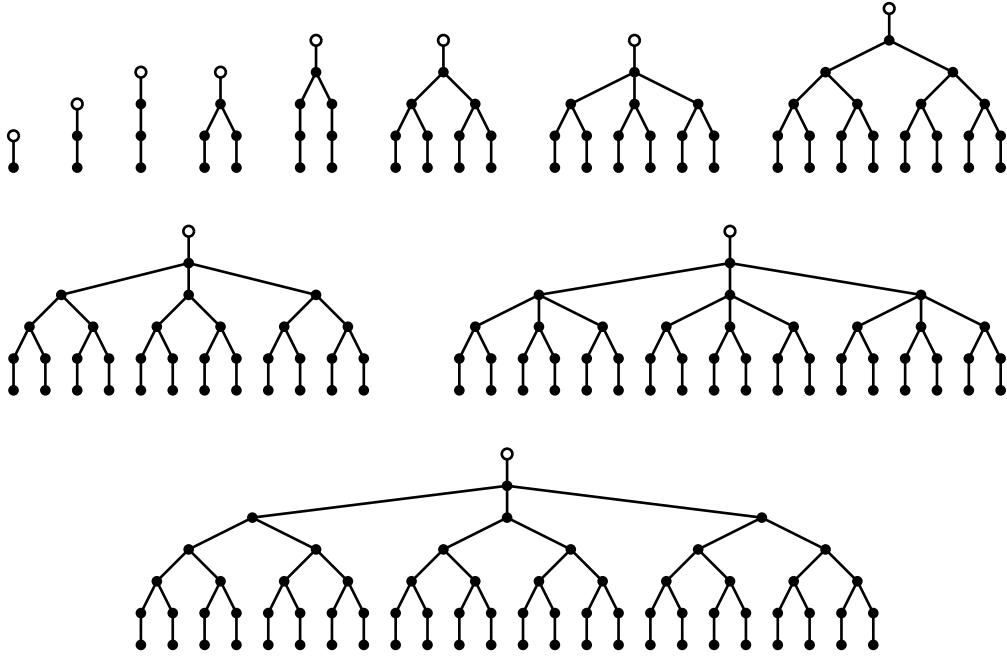}
	\caption{Beautiful trees $B_k$ with $f(k)$ edges for $1 \le k \le 11$.}\label{fig:f}
\end{figure}

\begin{proof}
	Clearly, $f(1) = 1$, and $B_1$ is the tree consisting of a single edge.

	For $k \ge 2$,
	let $T$ be a beautiful tree and height $h \ge 2$,
	with $\sum_{d = 0}^h c_d = k$,
	where $c_0 = 1$ and $1 \le c_1 \le 3$.
	Then the sequence
	$\langle c'_0,c'_1,\ldots,c'_{h-1} \rangle = \langle 1,c_2,\ldots,c_h \rangle$
	corresponds to a beautiful tree $T'$ of height $h-1$,
	with $\sum_{d = 0}^{h-1} c'_d = 1 + \sum_{d = 2}^h c_d = k - c_1$,
	which is a subtree of $T$.
	Indeed $T$ contains $c_1$ edge-disjoint copies of $T'$ with a shared root
	at the lower end of the single edge from the root of $T$.
	If $T$ has the maximum size $f(k)$, then $T'$ must have the maximum size $f(k - c_1)$.
	By enumerating $c_1$, we have the recurrence
	\begin{equation}\label{eq:f}
		f(k) = \max\{\, c_1\cdot f(k-c_1) + 1 \mid 1 \le c_1 < k \,\}.
	\end{equation}
	Moreover, since we can assume that $c_1 \le 3$,
	$$
	f(k) = \max\{\, c_1\cdot f(k-c_1) + 1 \mid 1 \le c_1 \le \min\{ k - 1, 3 \} \,\}.
	$$
	By the recurrence, we can derive $f(k)$ for $k = 2,\ldots,11$ sequentially:
	\begin{align*}
		f(1) &= 1,
		\\
		f(2) &= 1\cdot f(1) + 1 = 2,
		\\
		f(3) &= 1\cdot f(2) + 1 = 2\cdot f(1) + 1 = 3,
		\\
		f(4) &= 2\cdot f(2) + 1 = 5,
		\\
		f(5) &= 2\cdot f(3) + 1 = 3\cdot f(2) + 1 = 7,
		\\
		f(6) &= 2\cdot f(4) + 1 = 11,
		\\
		f(7) &= 3\cdot f(4) + 1 = 16,
		\\
		f(8) &= 2\cdot f(6) + 1 = 23,
		\\
		f(9) &= 3\cdot f(6) + 1 = 34,
		\\
		f(10) &= 3\cdot f(7) + 1 = 49,
		\\
		f(11) &= 3\cdot f(8) + 1 = 70.
	\end{align*}
	Let $\hat c(k)$ be the minimum $c_1$, $1 \le c_1 \le \min\{ k - 1, 3 \}$,
	such that $f(k) = c_1\cdot f(k - c_1) + 1$.
	Then $\hat c(k) = 1,1,2,2,2,3,2,3,3,3$ for $k = 2,\ldots,11$.
	In particular, $\hat c(k) = 3$ for $9 \le k \le 11$.
	Recall that a beautiful tree $T$ with sequence $\langle 1,c_1,c_2,\ldots,c_h \rangle$
	contains $c_1$ copies of a beautiful tree $T'$
	with sequence $\langle 1,c_2,\ldots,c_h \rangle$.
	Since $3 \ge c_1 \ge c_2$, we must have $\hat c(k) = 3$,
	and hence $f(k) = 3\cdot f(k - 3) + 1$, for $k \ge 9$.
	Thus we can determine $f(k)$ for $k \ge 6$ as follows.
\begin{itemize}\setlength\itemsep{0pt}
		\item
			For $k = 3i$ with $i \ge 2$,
			it follows from $f(6) = 11$ that
			$f(k) = \frac{23}2\cdot 3^{(k-6)/3} - \frac12$.

		\item
			For $k = 3i + 1$ with $i \ge 2$,
			it follows from $f(7) = 16$ that
			$f(k) = \frac{33}2\cdot 3^{(k-7)/3} - \frac12$.

		\item
			For $k = 3i + 2$ with $i \ge 2$,
			it follows from $f(8) = 23$ that
			$f(k) = \frac{47}2\cdot 3^{(k-8)/3} - \frac12$.
\end{itemize}

Refer to Figure~\ref{fig:f}.
For $k \ge 2$,
let $B_k$ be the beautiful tree with a single edge incident to the root,
whose lower end is the shared root of $c(k)$ copies of $B_{k'}$,
where $k' = k - \hat c(k)$.
Then the maximum size of a very hungry caterpillar in $B_k$ is exactly $k$.
Moreover,
since any caterpillar in $B_k$ can be covered by two very hungry caterpillars,
which always contain the single edge incident to the root,
the maximum size of any caterpillar in $B_k$ is at most $2k - 1$.

It remains to prove the two inequalities (a) and (b).

We first prove (a) that for $k \ge 2$,
$\frac{f(k)}{f(k - 1)} \ge \frac75$.
It is easy to verify this for $2 \le k \le 6$.
For $k \ge 7$, consider three cases:
\begin{itemize}\setlength\itemsep{0pt}
		\item
			For $k = 3i + 1$ with $i \ge 2$,
			$$
			\frac{f(k)}{f(k - 1)}
			= \frac{\frac{33}2\cdot 3^{(k-7)/3} - \frac12}{\frac{23}2\cdot 3^{(k-7)/3} - \frac12}
			= \frac{33\cdot 3^{(k-7)/3} - 1}{23\cdot 3^{(k-7)/3} - 1}
			\ge \frac{33\cdot 3^{(k-7)/3}}{23\cdot 3^{(k-7)/3}}
			= \frac{33}{23} > \frac75.
			$$
		\item
			For $k = 3i + 2$ with $i \ge 2$,
			$$
			\frac{f(k)}{f(k - 1)}
			= \frac{\frac{47}2\cdot 3^{(k-8)/3} - \frac12}{\frac{33}2\cdot 3^{(k-8)/3} - \frac12}
			= \frac{47\cdot 3^{(k-8)/3} - 1}{33\cdot 3^{(k-8)/3} - 1}
			\ge \frac{47\cdot 3^{(k-8)/3}}{33\cdot 3^{(k-8)/3}}
			= \frac{47}{33} > \frac75.
			$$
		\item
			For $k = 3i$ with $i \ge 3$,
			$$
			\frac{f(k)}{f(k - 1)}
			= \frac{\frac{23}2\cdot 3^{(k-6)/3} - \frac12}{\frac{47}2\cdot 3^{(k-9)/3} - \frac12}
			= \frac{69\cdot 3^{(k-9)/3} - 1}{47\cdot 3^{(k-9)/3} - 1}
			\ge \frac{69\cdot 3^{(k-9)/3}}{47\cdot 3^{(k-9)/3}}
			= \frac{69}{47} > \frac75.
			$$
\end{itemize}

We next prove (b) that for $k \ge 7$,
$\frac{f(k)}{f(k - 1)} < \frac32$.
Again consider three cases:
\begin{itemize}\setlength\itemsep{0pt}
		\item
			For $k = 3i + 1$ with $i \ge 2$,
			$$
			\frac{f(k)}{f(k - 1)}
			= \frac{33\cdot 3^{(k-7)/3} - 1}{23\cdot 3^{(k-7)/3} - 1}
			\le \frac{33\cdot 3^{(k-7)/3} - 3^{(k-7)/3}}{23\cdot 3^{(k-7)/3} - 3^{(k-7)/3}}
			= \frac{33 - 1}{23 - 1} < \frac32.
			$$
		\item
			For $k = 3i + 2$ with $i \ge 2$,
			$$
			\frac{f(k)}{f(k - 1)}
			= \frac{47\cdot 3^{(k-8)/3} - 1}{33\cdot 3^{(k-8)/3} - 1}
			\le \frac{47\cdot 3^{(k-8)/3} - 3^{(k-8)/3}}{33\cdot 3^{(k-8)/3} - 3^{(k-8)/3}}
			= \frac{47 - 1}{33 - 1} < \frac32.
			$$
		\item
			For $k = 3i$ with $i \ge 3$,
			$$
			\frac{f(k)}{f(k - 1)}
			= \frac{69\cdot 3^{(k-9)/3} - 1}{47\cdot 3^{(k-9)/3} - 1}
			\le \frac{69\cdot 3^{(k-9)/3} - 3^{(k-9)/3}}{47\cdot 3^{(k-9)/3} - 3^{(k-9)/3}}
			= \frac{69 - 1}{47 - 1} < \frac32.\qedhere
			$$
\end{itemize}
\end{proof}

\subsection{Multiple branches: $g(k)$}

Let $T$ be a tree with at least two edges,
and let $C$ be a caterpillar with the maximum number $k \ge 2$ edges in $T$.
Then there exists a non-leaf vertex $v$ in $C$,
of the same degree $r \ge 2$ in both $C$ and $T$,
that connects $r$ edge-disjoint parts of $C$,
including $r - 2$ single edges, and two smaller caterpillars
of $x \ge 1$ and $y \ge 1$ edges,
where $x + y + r - 2 = k$,
and $x \le y \le \lceil k/2 \rceil$.
These $r$ parts of $C$ are in $r$ subtrees of $T$, respectively;
in particular, the two smaller caterpillars
are very hungry caterpillars in two subtrees of $T$ rooted at $v$.
The size of the subtree containing the very hungry caterpillar of size $x$
is at most $f(x)$.
The size of the subtree containing the very hungry caterpillar of size $y$
is at most $f(y)$.
Moreover, each of the other $r-2$ subtrees can have size at most $f(x)$,
because otherwise we can replace the caterpillar of size $x$ by a larger caterpillar.
Thus the size of $T$ is at most $f(y) + (r - 1)\cdot f(x)$.

For $k \ge r \ge 2$, define
$$
g(r, k) = \max\{\, f(y) + (r - 1)\cdot f(x)
\mid x + y + r - 2 = k,\, 1 \le x \le y \le \lceil k/2 \rceil \,\}.
$$
For $k \ge 2$, define
$$
g(k) = \max\{\, g(r, k) \mid 2 \le r \le k \,\}.
$$
Then $g(k)$ is an upper bound on $e(k)$.
The following lemma shows that $e(k) = g(k)$ for $k \ge 2$.

\begin{lemma}\label{lem:g}
	$g(2) = 2$,
	$g(3) = 3$,
	$g(4) = 4$,
	$g(5) = 6$,
	$g(6) = 8$,
	$g(7) = 10$,
	$g(8) = 12$,
	$g(9) = 15$,
	$g(10) = 20$,
	$g(11) = 25$,
	$g(12) = 30$,
	$g(13) = 35$,
	$g(14) = 44$.
\begin{itemize}\setlength\itemsep{0pt}
		\item
			For $k = 2 i$ with $i \ge 8$,
			$g(k) = 6 f(\frac{k - 4}2)$.
		\item
			For $k = 2 i + 1$ with $i \ge 7$,
			$g(k) = 5 f(\frac{k - 3}2)$.
\end{itemize}
Moreover,
for $k \ge 2$, there exists a tree $T_k$ with exactly $g(k)$ edges
such that the maximum size of a caterpillar in $T_k$ is $k$.
\end{lemma}

\begin{proof}
	Write $h(r, x, y) = f(y) + (r - 1)\cdot f(x)$.
	Then
	$$
	g(k) = \max\{\, h(r, x, y) \mid
	2 \le r \le k,\,
	x + y + r - 2 = k,\,
	1 \le x \le y \le \lceil k/2 \rceil
	\,\}.
	$$
	With this formula, we can determine $g(k)$ for small values of $k$ without much difficulty.
	Incidentally, for each $k = 2,\ldots,14$, $g(k)$ can be realized by $h(r, x, y)$
	for some tuple $(r, x, y)$ with $x = y$:
	\begin{align*}
		g(2) &= h(2,1,1) = 2\cdot f(1) = 2,\\
		g(3) &= h(3,1,1) = 3\cdot f(1) = 3,\\
		g(4) &= h(4,1,1) = 4\cdot f(1) = 4,\\
		g(5) &= h(3,2,2) = 3\cdot f(2) = 6,\\
		g(6) &= h(4,2,2) = 4\cdot f(2) = 8,\\
		g(7) &= h(5,2,2) = 5\cdot f(2) = 10,\\
		g(8) &= h(4,3,3) = 4\cdot f(3) = 12,\\
		g(9) &= h(3,4,4) = 3\cdot f(4) = 15,\\
		g(10) &= h(4,4,4) = 4\cdot f(4) = 20,\\
		g(11) &= h(5,4,4) = 5\cdot f(4) = 25,\\
		g(12) &= h(6,4,4) = 6\cdot f(4) = 30,\\
		g(13) &= h(5,5,5) = 5\cdot f(5) = 35,\\
		g(14) &= h(4,6,6) = 4\cdot f(6) = 44.
	\end{align*}

	In the following, we assume that $k \ge 15$.
	Note that
	$5 f(\frac{k - 3}2) = h(5, \frac{k-3}2, \frac{k-3}2)$
	and
	$6 f(\frac{k - 4}2) = h(6, \frac{k-4}2, \frac{k-4}2)$.
	To find the maximum value of $h(r, x, y)$
	for
	$2 \le r \le k$,
	$x + y + r - 2 = k$,
	and
	$1 \le x \le y \le \lceil k/2 \rceil$,
	we can assume that either $x = y$, or $x < y$ and $2 y > k$,
	because if $x < y$ and $2 y \le k$,
	then at least one of the following two inequalities must hold:
	\begin{align}
		h(r, x, y) &\le h(r + (y - x), x, x),\label{eq:plus}\\
		h(r, x, y) &\le h(r - (y - x), y, y).\label{eq:minus}
	\end{align}

	Inequality~\eqref{eq:plus}
	is equivalent to
	\begin{align}\label{eq:ge}
		(r - 1) f(x) + f(y) &\le (r + (y - x)) f(x)\nonumber\\
		f(y) &\le (y - x + 1) f(x)\nonumber\\
		\frac{f(x)}{f(y)} &\ge \frac1{y - x + 1}.
	\end{align}

	Inequality~\eqref{eq:minus}
	is equivalent to
	\begin{align}\label{eq:le}
		(r - 1) f(x) + f(y) &\le (r - (y - x)) f(y)\nonumber\\
		(r - 1) f(x) &\le (r - 1 - (y - x)) f(y)\nonumber\\
		\frac{f(x)}{f(y)} &\le 1 - \frac{y - x}{r - 1}.
	\end{align}

	Since $x + y + r - 2 = k$,
	we have $r - 1 = k - x - y + 1 = k - 2y + y - x + 1$.
	If $2y \le k$, then
	$r - 1 \ge y - x + 1$,
	and hence
	$$
	1 - \frac{y - x}{r - 1}
	\ge
	1 - \frac{y - x}{y - x + 1} = \frac1{y - x + 1}.
	$$
	Thus at least one of \eqref{eq:ge} and \eqref{eq:le} must hold.
	It follows that at least one of \eqref{eq:plus} and \eqref{eq:minus} must hold.

	Therefore,
	to find the maximum value of $h(r, x, y)$
	for
	$2 \le r \le k$,
	$x + y + r - 2 = k$,
	and
	$1 \le x \le y \le \lceil k/2 \rceil$,
	we can assume that either $x = y$, or $x < y$ and $2 y > k$.
	Consider three cases:

	\bigskip
	Case 1:
	$2 \le r \le k$, $x + y + r - 2 = k$, $1 \le x = y \le \lceil k/2 \rceil$,
	$k \ge 16$ is even.
	Then $r$ is also even,
	and $x = y = \frac{k - r + 2}2$.
	We next show that
	$r f(\frac{k - r + 2}2) \le 6 f(\frac{k - 4}2)$
	for all even $r$, $2 \le r \le k$:
\begin{itemize}\setlength\itemsep{0pt}
		\item $r = 2$.
			Since $k \ge 16$, we have $\frac{k}2 \ge 8$,
			and it follows by Lemma~\ref{lem:f}(b) that
			$\frac{f(\frac{k}2)}{f(\frac{k - 4}2)} < (\frac32)^2 < \frac62$.
			Thus $2 f(\frac{k}2) < 6 f(\frac{k - 4}2)$.

		\item $r = 4$.
			Since $k \ge 16$, we have $\frac{k - 2}2 \ge 7$,
			and it follows by Lemma~\ref{lem:f}(b) that
			$\frac{f(\frac{k - 2}2)}{f(\frac{k - 4}2)} < \frac32 = \frac64$.
			Thus $4 f(\frac{k - 2}2) < 6 f(\frac{k - 4}2)$.

		\item $6 \le r \le k$.
			We prove by induction that $r f(\frac{k - r + 2}2) \le 6 f(\frac{k - 4}2)$.
			For the base case when $r = 6$,
			we have $r f(\frac{k - r + 2}2) = 6 f(\frac{k - 4}2)$.
			For the inductive step, fix $8 \le r \le k$.
			By Lemma~\ref{lem:f}(a), we have
			$\frac{f(\frac{k - r + 4}2)}{f(\frac{k - r + 2}2)} \ge \frac75 > \frac{r}{r - 2}$,
			and hence
			$r f(\frac{k - r + 2}2) < (r - 2) f(\frac{k - r + 4}2)$.
			By induction,
			$(r - 2) f(\frac{k - r + 4}2) \le 6 f(\frac{k - 4}2)$.
			Thus
			$r f(\frac{k - r + 2}2) < 6 f(\frac{k - 4}2)$.
\end{itemize}

\bigskip
Case 2:
$2 \le r \le k$, $x + y + r - 2 = k$, $1 \le x = y \le \lceil k/2 \rceil$,
$k \ge 15$ is odd.
Then $r$ is also odd,
and $x = y = \frac{k - r + 2}2$.
We next show that
$r f(\frac{k - r + 2}2) \le 5 f(\frac{k - 3}2)$
for all odd $r$, $3 \le r \le k$:
	\begin{itemize}\setlength\itemsep{0pt}
			\item $r = 3$.
				Since $k \ge 15$, we have $\frac{k - 1}2 \ge 7$,
				and it follows by Lemma~\ref{lem:f}(b) that
				$\frac{f(\frac{k - 1}2)}{f(\frac{k - 3}2)} < \frac32 < \frac53$.
				Thus $3 f(\frac{k - 1}2) < 5 f(\frac{k - 3}2)$.

			\item $5 \le r \le k$.
				We prove by induction that $r f(\frac{k - r + 2}2) \le 5 f(\frac{k - 3}2)$.
				For the base case when $r = 5$,
				we have $r f(\frac{k - r + 2}2) = 5 f(\frac{k - 3}2)$.
				For the inductive step, fix $7 \le r \le k$.
				By Lemma~\ref{lem:f}(a), we have
				$\frac{f(\frac{k - r + 4}2)}{f(\frac{k - r + 2}2)} \ge \frac75 \ge \frac{r}{r - 2}$,
				and hence
				$r f(\frac{k - r + 2}2) \le (r - 2) f(\frac{k - r + 4}2)$.
				By induction,
				$(r - 2) f(\frac{k - r + 4}2) \le 5 f(\frac{k - 3}2)$.
				Thus $r f(\frac{k - r + 2}2) \le 5 f(\frac{k - 3}2)$.
	\end{itemize}

	\bigskip
	Case 3:
	$2 \le r \le k$, $x + y + r - 2 = k$, $1 \le x < y \le \lceil k/2 \rceil$, $2y > k \ge 15$.
	Since $y \le \lceil k/2 \rceil$ and $2y > k$,
	we must have $y = \frac{k + 1}2$ for an odd $k$.
	Then $x = k + 1 - y - r + 1 = y - r + 1$.
	With fixed $y = \frac{k + 1}2$, we have
	\begin{align*}
		&\max\{\, f(y) + (r - 1)\cdot f(x) \mid
		2 \le r \le k,\, x + y + r - 2 = k,\, 1 \le x \le y \le \lceil k/2 \rceil \,\}\\
		&= 
		f(y) + \max\{\, (r - 1)\cdot f(y - r + 1) \mid 2 \le r \le k, 1 \le y - r + 1 < y \,\}\\
		&= 
		f(y) + \max\{\, (r - 1)\cdot f(y - (r - 1)) \mid 2 \le r \le y \,\}\\
		&= 
		f(y) + \max\{\, c\cdot f(y - c) + 1 \mid 1 \le c < y \,\} - 1\\
		&= 
		f(y) + f(y) - 1,
	\end{align*}
	where the last step follows from the recurrence~\eqref{eq:f}
	in the proof of Lemma~\ref{lem:f}.
	Since $k \ge 15$, we have $\frac{k + 1}2 \ge 8$,
	and it follows by Lemma~\ref{lem:f}(b)
	that
	$\frac{f(\frac{k + 1}2)}{f(\frac{k - 3}2)} < (\frac32)^2 < \frac52$.
	Thus, 
	$2 f(\frac{k + 1}2) - 1 < 2 f(\frac{k + 1}2) < 5 {f(\frac{k - 3}2)}$.

	We have proved that
	$g(k) = 5 f(\frac{k - 3}2)$
	for $k = 2 i + 1$ with $i \ge 7$,
	and
	$g(k) = 6 f(\frac{k - 4}2)$
	for $k = 2 i$ with $i \ge 8$.
	In particular, for $15 \le k \le 26$, we have
	\begin{align*}
		g(15) &= 5\cdot f(6) = 55,   &g(16) &= 6\cdot f(6) = 66,\\
		g(17) &= 5\cdot f(7) = 80,   &g(18) &= 6\cdot f(7) = 96,\\
		g(19) &= 5\cdot f(8) = 115,  &g(20) &= 6\cdot f(8) = 138,\\
		g(21) &= 5\cdot f(9) = 170,  &g(22) &= 6\cdot f(9) = 204,\\
		g(23) &= 5\cdot f(10) = 245, &g(24) &= 6\cdot f(10) = 294,\\
		g(25) &= 5\cdot f(11) = 350, &g(26) &= 6\cdot f(11) = 420.
	\end{align*}

	Let $T_1$ be the tree consisting of a single edge.
	For $k \ge 2$, we have $g(k) = r f(x)$
	for some $r$ and $x$ where $2 \le r \le k$ and $x = \frac{k - r + 2}2$.
	Let $T_k$ be the tree consisting of $r$ copies of $B_x$ sharing the root.
	Then the size of $T_k$ is exactly $g(k)$.
	There exists a caterpillar in $T_k$ with exactly $k$ edges,
	including $2x$ edges of two very hungry caterpillars in two copies of $B_x$,
	and $r - 2$ other edges incident to the root that are not in these two very hungry
	caterpillars.
	Any caterpillar in $T_k$ must include, either at most one, or at least two,
	of the $r$ edges incident to the root.
	In the first case, the caterpillar is contained completely in one copy of $B_x$,
	and hence can have at most $2x - 1 = k - r + 1 < k$ edges by Lemma~\ref{lem:f}.
	In the second case, the caterpillar can extend into at most two copies of $B_x$ composing $T_k$,
	and hence can have at most $2x + r - 2 = k$ edges.

	Refer to Figures~\ref{fig:g234567}, \ref{fig:g8910}, \ref{fig:g1112},
	\ref{fig:g1314},
	\ref{fig:g1516},
	\ref{fig:g1718},
	\ref{fig:g1920},
	\ref{fig:g2122},
	\ref{fig:g2324},
	\ref{fig:g2526}
	for illustrations of $T_k$ for $2 \le k \le 26$.
\end{proof}

\begin{figure}[htbp]
	\centering\includegraphics{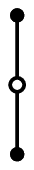}\includegraphics{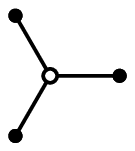}\includegraphics{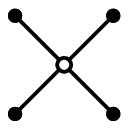}\includegraphics{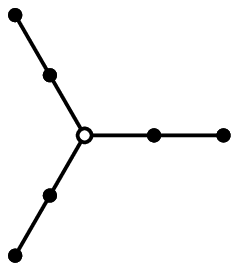}\includegraphics{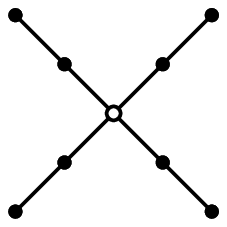}\includegraphics{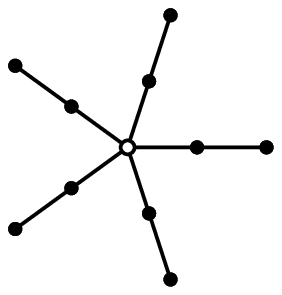}
	\caption{$T_2$ with $2$ edges, $T_3$ with $3$ edges, $T_4$ with $4$ edges,
	$T_5$ with $6$ edges, $T_6$ with $8$ edges, and $T_7$ with $10$ edges.}\label{fig:g234567}
\end{figure}

\begin{figure}[htbp]
	\centering\includegraphics{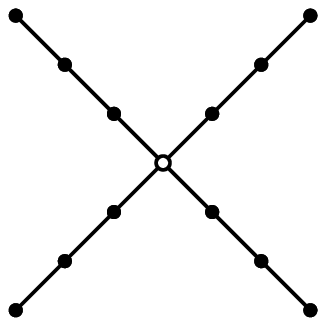}\quad\includegraphics{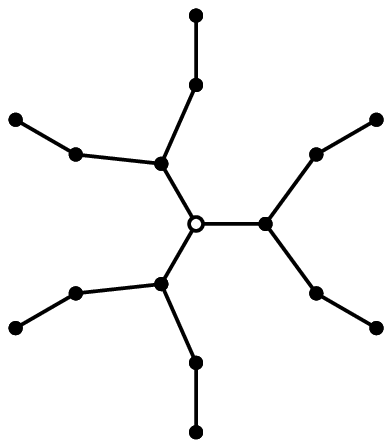}\quad\includegraphics{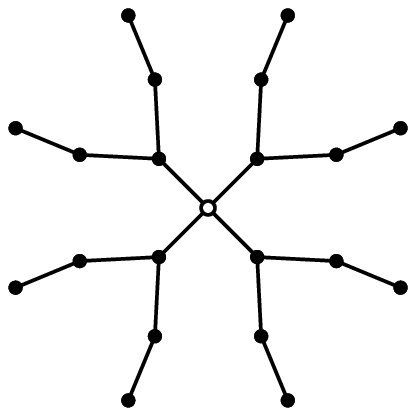}
	\caption{$T_8$ with $12$ edges, $T_9$ with $15$ edges, and $T_{10}$ with $20$ edges.}\label{fig:g8910}
\end{figure}

\begin{figure}[htbp]
	\centering\includegraphics{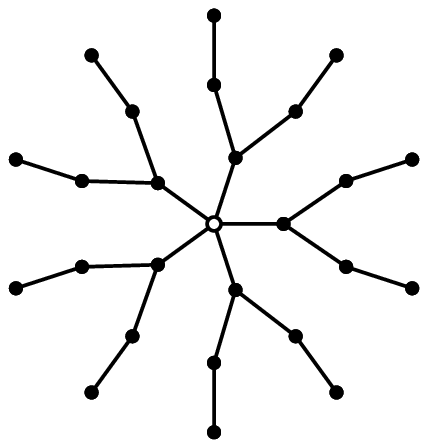}\qquad\includegraphics{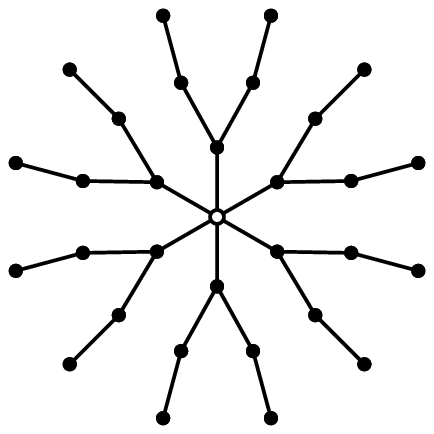}
	\caption{$T_{11}$ with $25$ edges and $T_{12}$ with $30$ edges.}\label{fig:g1112}
\end{figure}

\begin{figure}[htbp]
	\centering\includegraphics{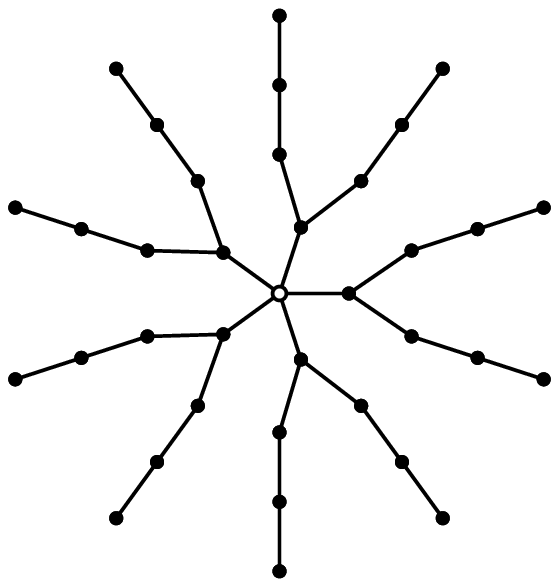}\qquad\includegraphics{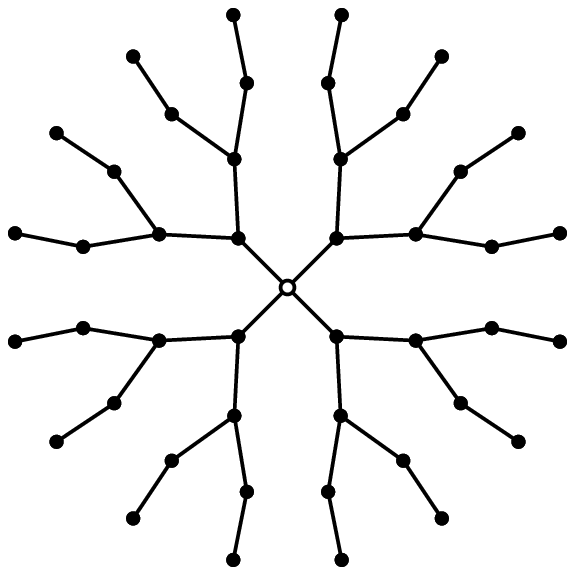}
	\caption{$T_{13}$ with $35$ edges and $T_{14}$ with $44$ edges.}\label{fig:g1314}
\end{figure}

\begin{figure}[htbp]
	\centering\includegraphics{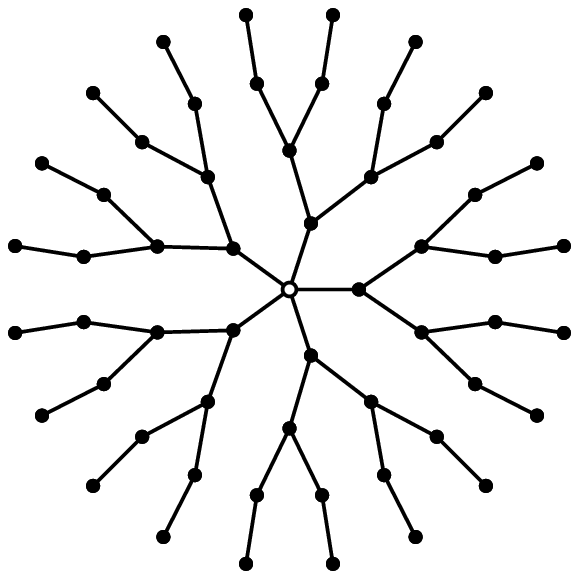}\qquad\includegraphics{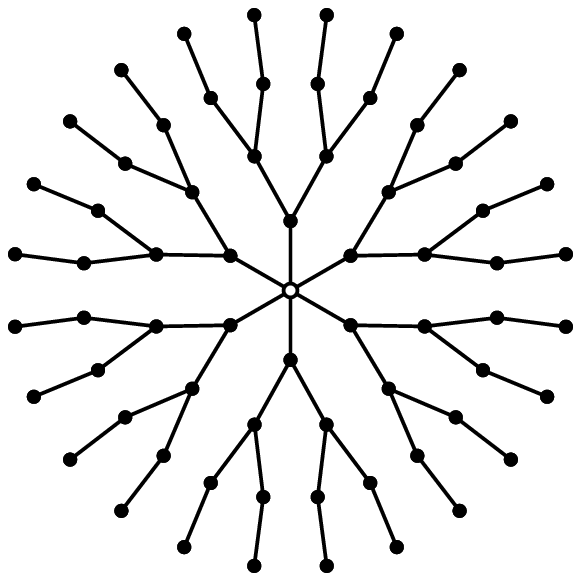}
	\caption{$T_{15}$ with $55$ edges and $T_{16}$ with $66$ edges.}\label{fig:g1516}
\end{figure}

\begin{figure}[htbp]
	\centering\includegraphics{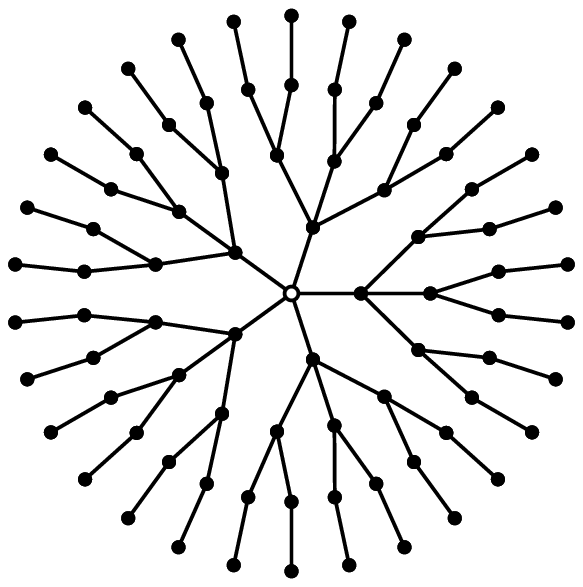}\qquad\includegraphics{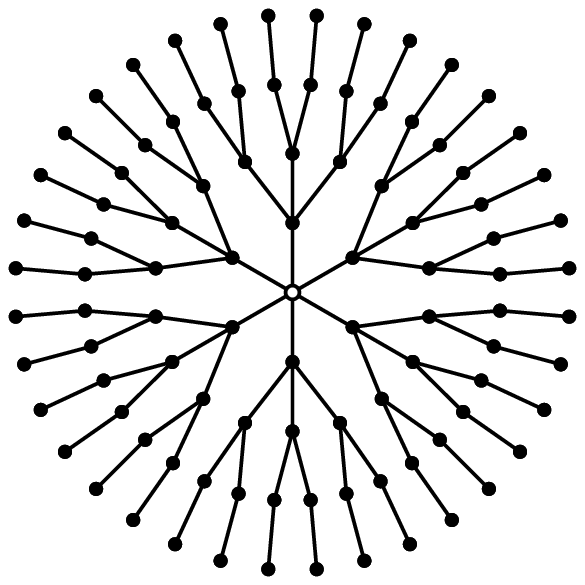}
	\caption{$T_{17}$ with $80$ edges and $T_{18}$ with $96$ edges.}\label{fig:g1718}
\end{figure}

\begin{figure}[htbp]
	\centering\includegraphics{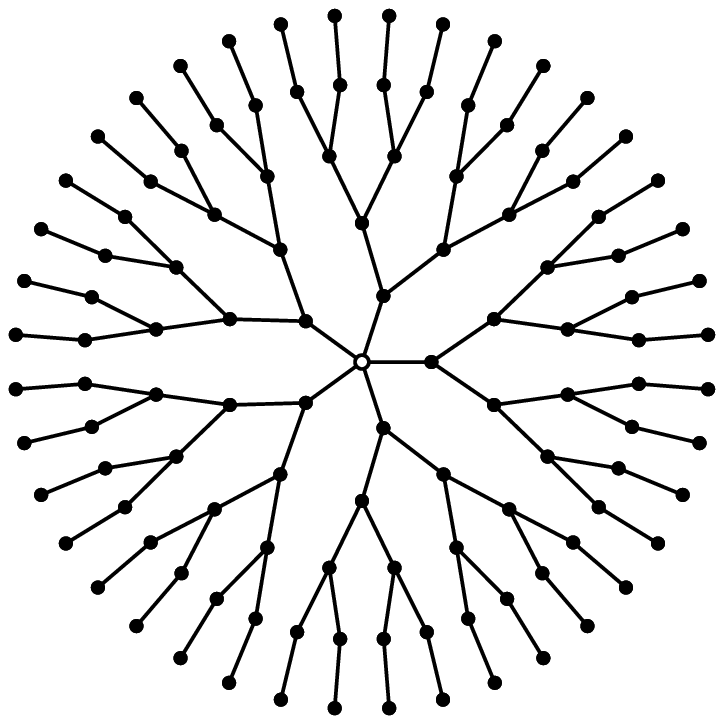}\qquad\includegraphics{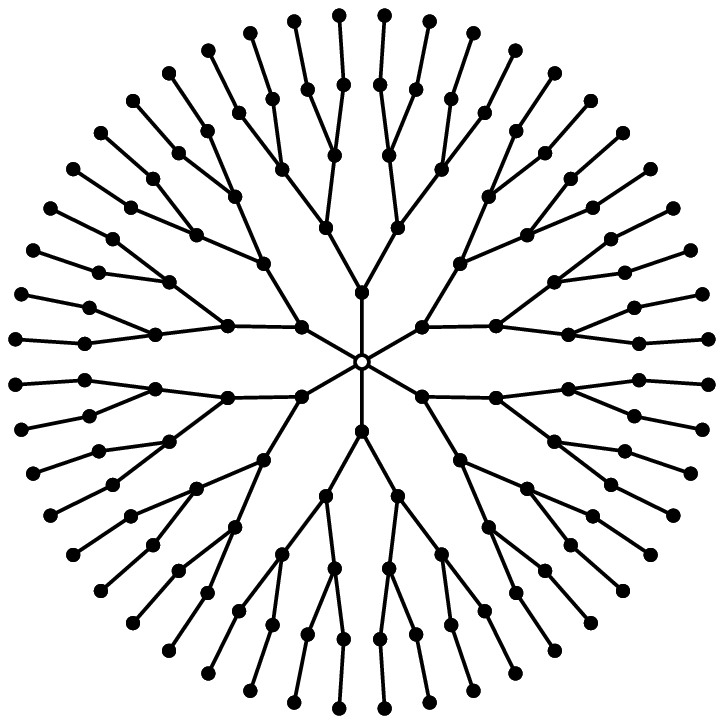}
	\caption{$T_{19}$ with $115$ edges and $T_{20}$ with $138$ edges.}\label{fig:g1920}
\end{figure}

\begin{figure}[htbp]
	\centering\includegraphics{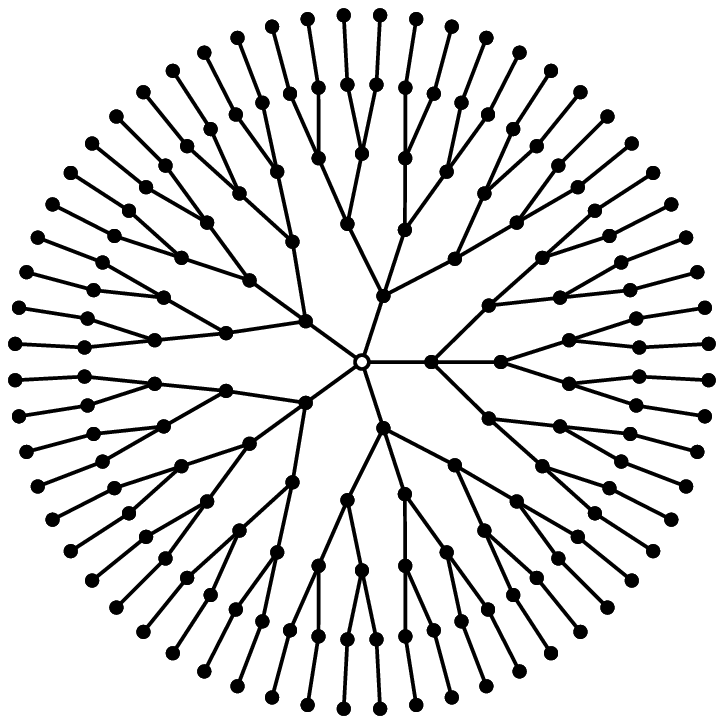}\qquad\includegraphics{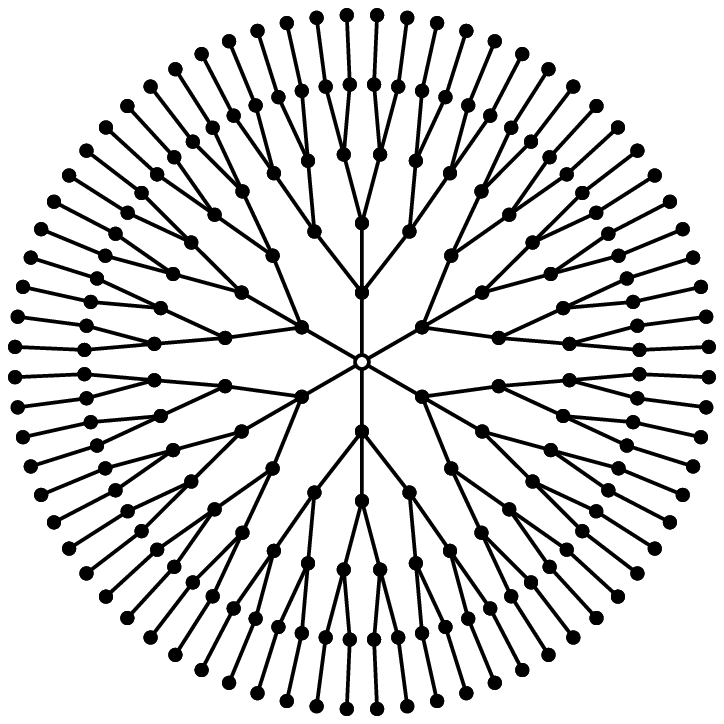}
	\caption{$T_{21}$ with $170$ edges and $T_{22}$ with $204$ edges.}\label{fig:g2122}
\end{figure}

\begin{figure}[htbp]
	\centering\includegraphics{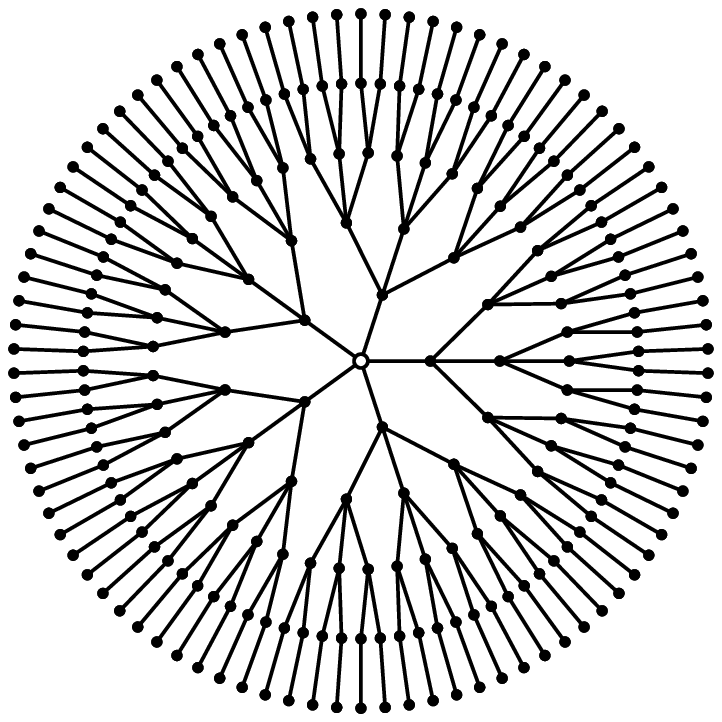}\qquad\includegraphics{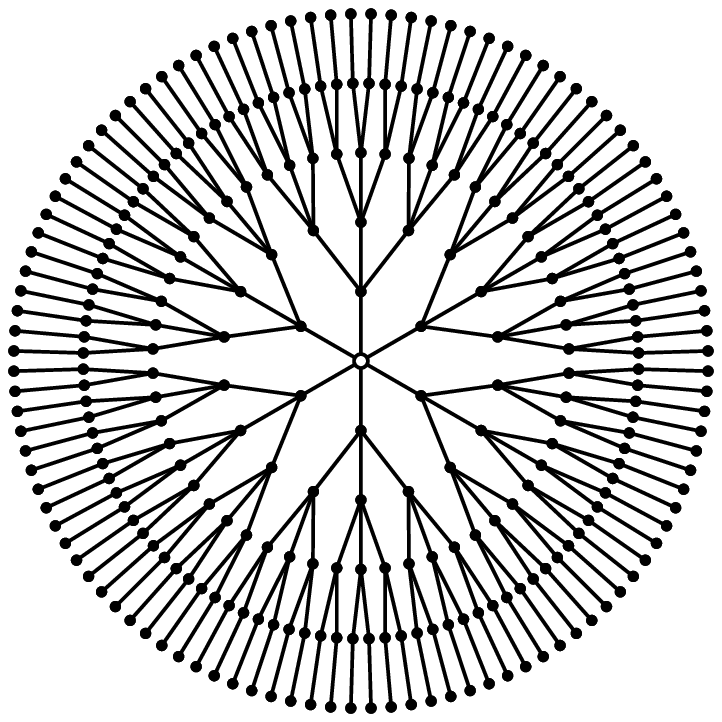}
	\caption{$T_{23}$ with $245$ edges and $T_{24}$ with $294$ edges.}\label{fig:g2324}
\end{figure}

\begin{figure}[htbp]
	\centering\resizebox{\linewidth}{!}{\includegraphics{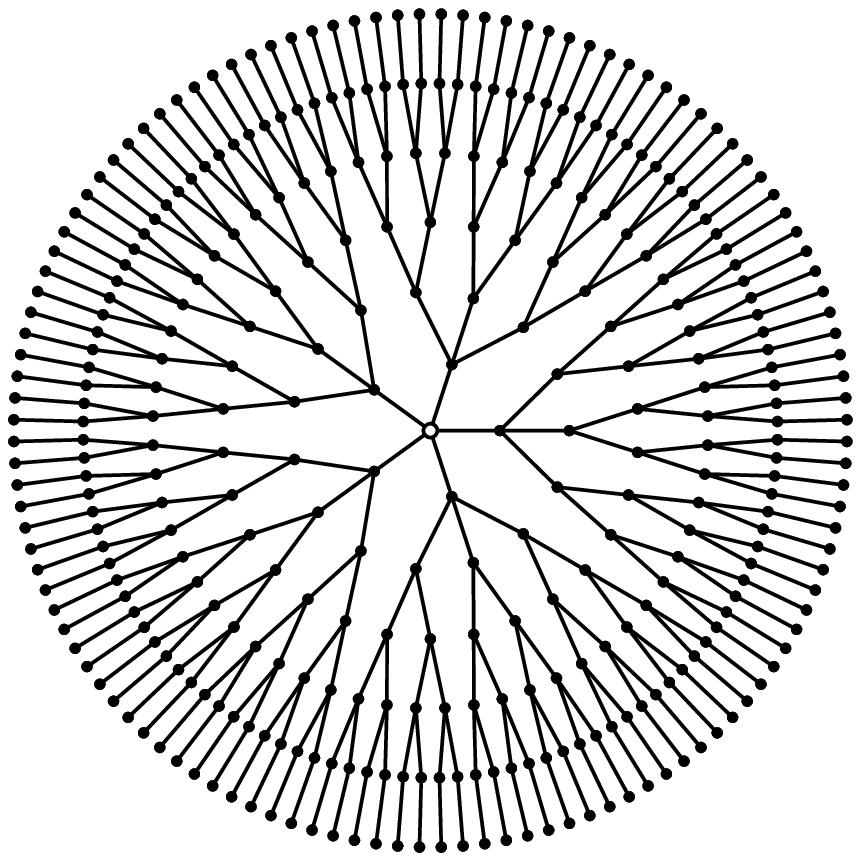}\includegraphics{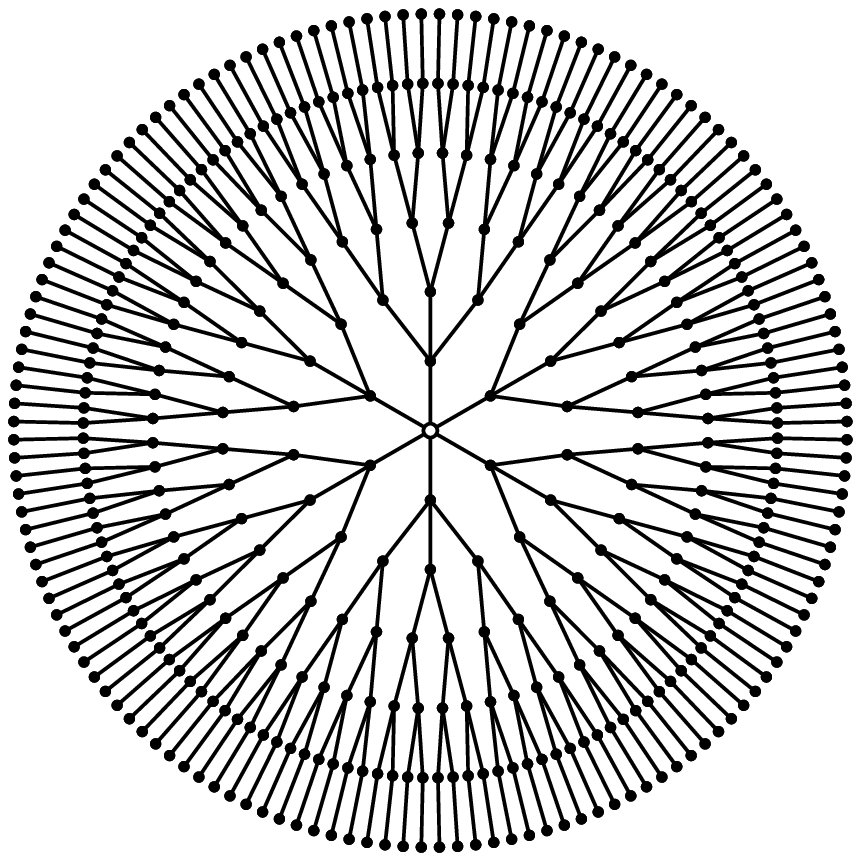}}
	\caption{$T_{25}$ with $350$ edges and $T_{26}$ with $420$ edges.}\label{fig:g2526}
\end{figure}

\newpage

\subsection{Deriving $e(k)$ from $f(k)$ and $g(k)$}

We clearly have $e(1) = 1$.
By Lemma~\ref{lem:g},
we have $e(k) = g(k)$ for $k \ge 2$.
Recall Lemma~\ref{lem:g} that
\begin{itemize}\setlength\itemsep{0pt}
		\item
			For $k = 2 i$ for $i \ge 8$,
			$g(k) = 6 f(\frac{k-4}2)$.
		\item
			For $k = 2 i + 1$ for $i \ge 7$,
			$g(k) = 5 f(\frac{k-3}2)$.
\end{itemize}

Also recall Lemma~\ref{lem:f} that
\begin{itemize}\setlength\itemsep{0pt}
		\item
			For $k = 3i$ with $i \ge 2$,
			$f(k) = \frac12(23\cdot 3^{(k-6)/3} - 1)$.
		\item
			For $k = 3i + 1$ with $i \ge 2$,
			$f(k) = \frac12(33\cdot 3^{(k-7)/3} - 1)$.
		\item
			For $k = 3i + 2$ with $i \ge 2$,
			$f(k) = \frac12(47\cdot 3^{(k-8)/3} - 1)$.
\end{itemize}

Altogether, there are six cases for $k \ge 15$:
\begin{itemize}\setlength\itemsep{0pt}
		\item
			For $k = 6i$ with $i \ge 3$,
			$e(k) = 6 f(\frac{k-4}2)
			= 3(33\cdot 3^{(\frac{k-4}2 - 7)/3} - 1)
			= 3(33\cdot 3^{(k-18)/6} - 1)$.
		\item
			For $k = 6i + 2$ with $i \ge 3$,
			$e(k) = 6 f(\frac{k-4}2)
			= 3(47\cdot 3^{(\frac{k-4}2 - 8)/3} - 1)
			= 3(47\cdot 3^{(k-20)/6} - 1)$.
		\item
			For $k = 6i + 4$ with $i \ge 2$,
			$e(k) = 6 f(\frac{k-4}2)
			= 3(23\cdot 3^{(\frac{k-4}2 - 6)/3} - 1)
			= 3(23\cdot 3^{(k-16)/6} - 1)$.
		\item
			For $k = 6i + 1$ with $i \ge 3$,
			$e(k) = 5 f(\frac{k-3}2)
			= \frac52(47\cdot 3^{(\frac{k-3}2 - 8)/3} - 1)
			= \frac52(47\cdot 3^{(k-19)/6} - 1)$.
		\item
			For $k = 6i + 3$ with $i \ge 2$,
			$e(k) = 5 f(\frac{k-3}2)
			= \frac52(23\cdot 3^{(\frac{k-3}2 - 6)/3} - 1)
			= \frac52(23\cdot 3^{(k-15)/6} - 1)$.
		\item
			For $k = 6i + 5$ with $i \ge 2$,
			$e(k) = 5 f(\frac{k-3}2)
			= \frac52(33\cdot 3^{(\frac{k-3}2 - 7)/3} - 1)
			= \frac52(33\cdot 3^{(k-17)/6} - 1)$.
\end{itemize}

In particular, for $k = 15,\ldots,21$, we have
\begin{align*}
	e(15) &= \frac52(23\cdot 3^{(15-15)/6} - 1) = 55,\\
	e(16) &= 3(23\cdot 3^{(16-16)/6} - 1) = 66,\\
	e(17) &= \frac52(33\cdot 3^{(17-17)/6} - 1) = 80,\\
	e(18) &= 3(33\cdot 3^{(18-18)/6} - 1) = 96,\\
	e(19) &= \frac52(47\cdot 3^{(19-19)/6} - 1) = 115,\\
	e(20) &= 3(47\cdot 3^{(20-20)/6} - 1) = 138,\\
	e(21) &= \frac52(23\cdot 3^{(21-15)/6} - 1) = 170.
\end{align*}

In summary, we have the following theorem:

\begin{theorem}\label{thm:e}
	$e(1) = 1$,
	$e(2) = 2$,
	$e(3) = 3$,
	$e(4) = 4$,
	$e(5) = 6$,
	$e(6) = 8$,
	$e(7) = 10$,
	$e(8) = 12$,
	$e(9) = 15$,
	$e(10) = 20$,
	$e(11) = 25$,
	$e(12) = 30$,
	$e(13) = 35$,
	$e(14) = 44$,
	$e(15) = 55$,
	$e(16) = 66$,
	$e(17) = 80$,
	$e(18) = 96$,
	$e(19) = 115$,
	$e(20) = 138$,
	$e(21) = 170$.
\begin{itemize}\setlength\itemsep{0pt}
		\item
			For $k = 6i$ with $i \ge 3$,
			$e(k) = 3(11\cdot 3^{(k-12)/6} - 1)$.
		\item
			For $k = 6i + 2$ with $i \ge 3$,
			$e(k) = 3(47\cdot 3^{(k-20)/6} - 1)$.
		\item
			For $k = 6i + 4$ with $i \ge 2$,
			$e(k) = 3(23\cdot 3^{(k-16)/6} - 1)$.
		\item
			For $k = 6i + 1$ with $i \ge 3$,
			$e(k) = \frac52(47\cdot 3^{(k-19)/6} - 1)$.
		\item
			For $k = 6i + 3$ with $i \ge 2$,
			$e(k) = \frac52(23\cdot 3^{(k-15)/6} - 1)$.
		\item
			For $k = 6i + 5$ with $i \ge 2$,
			$e(k) = \frac52(11\cdot 3^{(k-11)/6} - 1)$.
\end{itemize}
\end{theorem}

\subsection{Deriving $q(m)$ from $e(k)$}

By Theorem~\ref{thm:e},
the exact values of $e(k)$ for $1 \le k \le 21$
imply the following exact values of $q(m)$ for $1 \le m \le 170$:
\begin{center}
	\begin{tabular}{c|c c c c c c c c c}
		\hline
		$m$ & $[1,4]$ & $[5,6]$ & $[7,8]$ & $[9,10]$ & $[11,12]$ & $[13,15]$ & $[16,20]$ & $[21,25]$ & $[26,30]$\\
		$q(m)$ & $m$ & $5$ & $6$ & $7$ & $8$ & $9$ & $10$ & $11$ & $12$\\
		\hline
		$m$ & $[31,35]$ & $[36,44]$ & $[45,55]$ & $[56,66]$ & $[67,80]$ & $[81,96]$ & $[97,115]$ & $[116,138]$ & $[139,170]$\\
		$q(m)$ & $13$ & $14$ & $15$ & $16$ & $17$ & $18$ & $19$ & $20$ & $21$\\
		\hline
	\end{tabular}
\end{center}

In the following, we assume that $m \ge 171$.
If a tree $T$ has $m > e(k - 1)$ edges,
then it contains a caterpillar of $k$ edges.
Thus $q(m)$ is the largest integer $k$ such that $e(k - 1) < m$.

For $k - 1 = 6i$ with $i \ge 3$,
$e(k - 1) = 3(11\cdot 3^{(k-13)/6} - 1)$.
The inequality $e(k - 1) < m$ is equivalent to
\begin{gather*}
	3(11\cdot 3^{(k-13)/6} - 1) < m\\
	11\cdot 3^{(k-13)/6} < \frac{m}3 + 1\\
	3^{(k-13)/6} < \frac{m}{33} + \frac1{11}\\
	(k-13)/6 < \log_3\left( \frac{m}{33} + \frac1{11} \right)\\
	k < 6\log_3\left( \frac{m}{33} + \frac1{11} \right) + 13\\
	k < \left\lceil 6\log_3\left( \frac{m}{33} + \frac1{11} \right) + 13 \right\rceil\\
	k \le \left\lceil 6\log_3\left( \frac{m}{33} + \frac1{11} \right) + 12 \right\rceil\\
	k - 1 \le \left\lceil 6\log_3\left( \frac{m}{33} + \frac1{11} \right) + 11 \right\rceil.
\end{gather*}
Thus the largest integer $k$ such that $e(k - 1) < m$ and $k \bmod 6 = 1$ is
\begin{equation}\label{eq:q1}
	q_1(m) = 6\left\lfloor\frac16\left\lceil
	6\log_3\left( \frac{m}{33} + \frac1{11} \right) + 11
	\right\rceil\right\rfloor + 1.
\end{equation}

\newpage

For $k - 1 = 6i + 2$ with $i \ge 3$,
$e(k - 1) = 3(47\cdot 3^{(k-21)/6} - 1)$.
The inequality $e(k - 1) < m$ is equivalent to
\begin{gather*}
	3(47\cdot 3^{(k-21)/6} - 1) < m\\
	47\cdot 3^{(k-21)/6} < \frac{m}3 + 1\\
	3^{(k-21)/6} < \frac{m}{141} + \frac1{47}\\
	(k-21)/6 < \log_3\left( \frac{m}{141} + \frac1{47} \right)\\
	k < 6\log_3\left( \frac{m}{141} + \frac1{47} \right) + 21\\
	k < \left\lceil 6\log_3\left( \frac{m}{141} + \frac1{47} \right) + 21 \right\rceil\\
	k \le \left\lceil 6\log_3\left( \frac{m}{141} + \frac1{47} \right) + 20 \right\rceil\\
	k - 3 \le \left\lceil 6\log_3\left( \frac{m}{141} + \frac1{47} \right) + 17 \right\rceil.
\end{gather*}
Thus the largest integer $k$ such that $e(k - 1) < m$ and $k \bmod 6 = 3$ is
\begin{equation}\label{eq:q3}
	q_3(m) = 6\left\lfloor\frac16\left\lceil
	6\log_3\left( \frac{m}{141} + \frac1{47} \right) + 17
	\right\rceil\right\rfloor + 3.
\end{equation}

For $k - 1 = 6i + 4$ with $i \ge 2$,
$e(k - 1) = 3(23\cdot 3^{(k-17)/6} - 1)$.
The inequality $e(k - 1) < m$ is equivalent to
\begin{gather*}
	3(23\cdot 3^{(k-17)/6} - 1) < m\\
	23\cdot 3^{(k-17)/6} < \frac{m}3 + 1\\
	3^{(k-17)/6} < \frac{m}{69} + \frac1{23}\\
	(k-17)/6 < \log_3\left( \frac{m}{69} + \frac1{23} \right)\\
	k < 6\log_3\left( \frac{m}{69} + \frac1{23} \right) + 17\\
	k < \left\lceil 6\log_3\left( \frac{m}{69} + \frac1{23} \right) + 17 \right\rceil\\
	k \le \left\lceil 6\log_3\left( \frac{m}{69} + \frac1{23} \right) + 16 \right\rceil\\
	k - 5 \le \left\lceil 6\log_3\left( \frac{m}{69} + \frac1{23} \right) + 11 \right\rceil.
\end{gather*}
Thus the largest integer $k$ such that $e(k - 1) < m$ and $k \bmod 6 = 5$ is
\begin{equation}\label{eq:q5}
	q_5(m) = 6\left\lfloor\frac16\left\lceil
	6\log_3\left( \frac{m}{69} + \frac1{23} \right) + 11
	\right\rceil\right\rfloor + 5.
\end{equation}

For $k - 1 = 6i + 1$ with $i \ge 3$,
$e(k - 1) = \frac52(47\cdot 3^{(k-20)/6} - 1)$.
The inequality $e(k - 1) < m$ is equivalent to
\begin{gather*}
	\frac52(47\cdot 3^{(k-20)/6} - 1) < m\\
	47\cdot 3^{(k-20)/6} < \frac{2m}5 + 1\\
	3^{(k-20)/6} < \frac{2m}{235} + \frac1{47}\\
	(k-20)/6 < \log_3\left( \frac{2m}{235} + \frac1{47} \right)\\
	k < 6\log_3\left( \frac{2m}{235} + \frac1{47} \right) + 20\\
	k < \left\lceil 6\log_3\left( \frac{2m}{235} + \frac1{47} \right) + 20 \right\rceil\\
	k \le \left\lceil 6\log_3\left( \frac{2m}{235} + \frac1{47} \right) + 19 \right\rceil\\
	k - 2 \le \left\lceil 6\log_3\left( \frac{2m}{235} + \frac1{47} \right) + 17 \right\rceil.
\end{gather*}
Thus the largest integer $k$ such that $e(k - 1) < m$ and $k \bmod 6 = 2$ is
\begin{equation}\label{eq:q2}
	q_2(m) = 6\left\lfloor\frac16\left\lceil
	6\log_3\left( \frac{2m}{235} + \frac1{47} \right) + 17
	\right\rceil\right\rfloor + 2.
\end{equation}

For $k - 1 = 6i + 3$ with $i \ge 2$,
$e(k - 1) = \frac52(23\cdot 3^{(k-16)/6} - 1)$.
The inequality $e(k - 1) < m$ is equivalent to
\begin{gather*}
	\frac52(23\cdot 3^{(k-16)/6} - 1) < m\\
	23\cdot 3^{(k-16)/6} < \frac{2m}5 + 1\\
	3^{(k-16)/6} < \frac{2m}{115} + \frac1{23}\\
	(k-16)/6 < \log_3\left( \frac{2m}{115} + \frac1{23} \right)\\
	k < 6\log_3\left( \frac{2m}{115} + \frac1{23} \right) + 16\\
	k < \left\lceil 6\log_3\left( \frac{2m}{115} + \frac1{23} \right) + 16 \right\rceil\\
	k \le \left\lceil 6\log_3\left( \frac{2m}{115} + \frac1{23} \right) + 15 \right\rceil\\
	k - 4 \le \left\lceil 6\log_3\left( \frac{2m}{115} + \frac1{23} \right) + 11 \right\rceil.
\end{gather*}
Thus the largest integer $k$ such that $e(k - 1) < m$ and $k \bmod 6 = 4$ is
\begin{equation}\label{eq:q4}
	q_4(m) = 6\left\lfloor\frac16\left\lceil
	6\log_3\left( \frac{2m}{115} + \frac1{23} \right) + 11
	\right\rceil\right\rfloor + 4.
\end{equation}

For $k - 1 = 6i + 5$ with $i \ge 2$,
$e(k - 1) = \frac52(11\cdot 3^{(k-12)/6} - 1)$.
The inequality $e(k - 1) < m$ is equivalent to
\begin{gather*}
	\frac52(11\cdot 3^{(k-12)/6} - 1) < m\\
	11\cdot 3^{(k-12)/6} < \frac{2m}5 + 1\\
	3^{(k-12)/6} < \frac{2m}{55} + \frac1{11}\\
	(k-12)/6 < \log_3\left( \frac{2m}{55} + \frac1{11} \right)\\
	k < 6\log_3\left( \frac{2m}{55} + \frac1{11} \right) + 12\\
	k < \left\lceil 6\log_3\left( \frac{2m}{55} + \frac1{11} \right) + 12 \right\rceil\\
	k \le \left\lceil 6\log_3\left( \frac{2m}{55} + \frac1{11} \right) + 11 \right\rceil.
\end{gather*}
Thus the largest integer $k$ such that $e(k - 1) < m$ and $k \bmod 6 = 0$ is
\begin{equation}\label{eq:q0}
	q_0(m) = 6\left\lfloor\frac16\left\lceil
	6\log_3\left( \frac{2m}{55} + \frac1{11} \right) + 11
	\right\rceil\right\rfloor.
\end{equation}

Thus for $m \ge 171$,
$q(m) = \max\{\, q_r(m) \mid 0 \le r \le 5 \,\}$.
This completes the proof of Theorem~\ref{thm:conv-log}.

\section{An open question}

Recall our definitions of
$\hat p(n)$, $\hat q(n)$, $p(n)$, and $q(n)$ for $n \ge 1$:
\begin{itemize}\setlength\itemsep{0pt}
		\item
			$\hat p(n)$ (respectively, $\hat q(n)$)
			is the maximum number $k$
			such that any family $\S$ of $n$ disjoint segments in the plane
			admits an alternating path among $\S$ (respectively, compatible with $\S$)
			going through $k$ segments in $\S$,
		\item
			$p(n)$ (respectively, $q(n)$)
			is the maximum number $k$
			such that any family $\S$ of $n$ disjoint segments in the plane,
			whose $2n$ endpoints are in convex position,
			admits an alternating path among $\S$ (respectively, compatible with $\S$)
			going through $k$ segments in $\S$.
\end{itemize}
We clearly have $\hat p(n) \le p(n)$ and $\hat q(n) \le q(n)$ for all $n \ge 1$,
but can we have
$\hat p(n) < p(n)$ or $\hat q(n) < q(n)$
for some $n \ge 1$?
In other words,
could it be true that
$\hat p(n) = p(n)$ and $\hat q(n) = q(n)$ for all $n \ge 1$?

\end{document}